\documentclass[a4paper,12pt]{amsart}
\usepackage[utf8]{inputenc}
\usepackage{bbm}
\usepackage{amsfonts}
\usepackage{amssymb}
\usepackage{amsthm}
\usepackage{newtxtext,newtxmath}
\usepackage{centernot}
\usepackage{mathtools}
\usepackage{fullpage}
\usepackage{hyperref}
\usepackage{accents}
\usepackage{color}

\let\originalleft\left
\let\originalright\right
\renewcommand{\left}{\mathopen{}\mathclose\bgroup\originalleft}
\renewcommand{\right}{\aftergroup\egroup\originalright}
\newcommand{\relmiddle}[1]{\mathrel{}\middle#1\mathrel{}}
\newcommand{\cond}{\relmiddle{|}}

\newcommand{\N}{\mathbb N}
\newcommand{\1}{\mathbbm 1}
\newcommand{\Z}{\mathbb Z}
\newcommand{\R}{\mathbb R}
\renewcommand{\P}{\mathbb{P}_p}
\newcommand{\Pq}{\mathbb{P}_q}
\newcommand{\E}{\mathbb{E}_p}
\newcommand{\Eq}{\mathbb{E}_q}
\newcommand{\F}{\mathcal F}
\newcommand{\eps}{\varepsilon}
\renewcommand{\epsilon}{\varepsilon}

\newtheorem{theorem}{Theorem}[section]
\newtheorem{lemma}[theorem]{Lemma}
\newtheorem{proposition}[theorem]{Proposition}

\theoremstyle{definition}
\newtheorem{definition}[theorem]{Definition}
\newtheorem{remark}[theorem]{Remark}

\newcommand{\cref}[1]{\text{\upshape\ref{#1}}}

\newcommand{\pcr}{\vec p_{\textup{cr}}(\Z^d)}
\newcommand{\C}{C}

\newcommand{\A}{A^{\textup{conn}}}
\renewcommand{\AA}{\accentset{\leftarrow}{A}^{\textup{conn}}} 
\newcommand{\AAA}{A^{\textup{repair}}}

\newcommand{\Ap}{{B}^{\textup{conn}}}
\newcommand{\AAp}{\accentset{\leftarrow}{B}^{\textup{conn}}} 
\newcommand{\AAAp}{B^{\textup{repair}}}

\newcommand{\omegas}{{\omega}^{\textup{b}}}
\newcommand{\Ps}{\P^{\,\textup{b}}}

\newcommand{\omegass}{{\widetilde\omega}^{\,\textup{b}}}
\newcommand{\Pss}{\widetilde{\mathbb P}_p^{\textup{b}}}

\newcommand{\omegae}{\omega^{\textup{e}}}
\newcommand{\Pe}{\P^{\,\textup{e}}}
\newcommand{\Ee}{\E^{\textup{e}}}

\newcommand{\omegasl}{\omega^{\textup{sl}}}
\newcommand{\Psl}{\P^{\,\textup{sl}}}
\newcommand{\Esl}{\E^{\,\textup{sl}}}

\newcommand{\Psle}{\P^{\,\textup{sl,e}}}
\newcommand{\Esle}{\E^{\,\textup{sl,e}}}

\newcommand{\G}{G}
\newcommand{\GG}{\accentset{\leftarrow}{G}}
\newcommand{\GGG}{\widehat G}

\newcommand{\I}{\accentset{\to}I}
\newcommand{\J}{\accentset{\leftarrow}{I}}

\newcommand{\f}{\mathfrak f}

\newcommand{\tEa}{\widetilde{E}_1}
\newcommand{\tEb}{\widetilde{E}_2}
\newcommand{\tEc}{\widetilde{E}_3}

\renewcommand{\nleftrightarrow}{\centernot\leftrightarrow}

\begin{document}
\title{Number of paths in oriented percolation as zero temperature limit of directed polymer}
\author[R.~Fukushima]{Ryoki Fukushima}
\address{Institute of Mathematics, University of Tsukuba, 1-1-1 Tennodai, Tsukuba, Ibaraki 305--8571, Japan}
\email{ryoki@math.tsukuba.ac.jp}

\author[S.~Junk]{Stefan Junk}
\address{Institute of Mathematics, University of Tsukuba, 1-1-1 Tennodai, Tsukuba, Ibaraki 305--8571, Japan}
\curraddr{Advanced Institute for Materials Research Tohoku University Mathematical Group, 2-1-1 Katahira, Aoba-ku, Sendai, 980-8577 Japan}
\email{sjunk@tohoku.ac.jp}

\begin{abstract}
We prove that the free energy of directed polymer in Bernoulli environment converges to the growth rate for the number of open paths in super-critical oriented percolation as the temperature tends to zero. Our proof is based on rate of convergence results which hold uniformly in the temperature. We also prove that the convergence rate is locally uniform in the percolation parameter inside the super-critical phase, which implies that the growth rate depends continuously on the percolation parameter. 
\end{abstract}

\maketitle

\section{Introduction}
This article concerns the number of open paths in oriented percolation on $\Z_+\times \Z^d$. The question on the existence of an infinite open path is well-studied and it is well-known that there is a phase transition for all $d\ge 1$. The question on the number of open paths in the super-critical phase dates back to the work by Darling~\cite{Dar91}, but it has received a renewed interest in recent years. 
Among others, it was proved by Garet--Gou\'er\'e--Marchand in~\cite{GGM17} that the number of open paths of length $n$, starting from a percolation point, grows like $e^{\tilde\alpha_p n+o(n)}$ for some deterministic $\tilde{\alpha}_p>0$ in the super-critical phase. In fact, they proved the existence of the growth rate for any fixed direction and $\tilde{\alpha}_p$ is the supremum of the directional growth rates. However, since the growth rate $\tilde\alpha_p$ is found by using an ergodic theorem, it remains a non-trivial task to study its properties as a function of the direction and the percolation parameter. It is shown to be a concave function of the direction but strict concavity has not been proved. As for the dependence on the percolation parameter, even the continuity has not been proved. 

In this article, we establish two approximation results for $\tilde{\alpha}_p$:
\begin{itemize}
 \item The first is the finite volume approximation with a quantitative error control. 
 \item The second is the approximation by the positive temperature version, which is the free energy for the directed polymer in the Bernoulli random environment.
\end{itemize}
These results aim at better understanding of $\tilde{\alpha}_p$. In particular, we show the continuity of $\tilde{\alpha}_p$ in $p\in(\pcr,1]$ as a corollary to the first result. 

\section{Setting and known results}
To fix the idea, we focus on the nearest neighbor site percolation model, while the results easily extend to bond percolation and to finite range generalizations. The reason for this choice is that the directed polymer models are mostly studied in the site disorder setting in the literature. Let $(\omega=(\omega(t,x))_{(t,x)\in \Z\times \Z^d},\P)$ be a collection of independent Bernoulli random variables with parameter $1-p$. The time-space point $(t,x)$ is said to be \emph{open} if $\omega(t,x)=0$ and \emph{closed} otherwise. This goes against the convention in the percolation literature but makes the relation to the directed polymer simple. For non-negative integers $m\le n$, let $\F_{[m,n]}$ be the $\sigma$-field generated by $(\omega(t,x))_{(t,x)\in [m,n]\times \Z^d}$ and write $\F_n=\F_{[0,n]}$ for short. A \emph{path} is a function $\pi$ from $\{s,s+1,\ldots,t\}\subset \Z$ to $\Z^d$ with $|\pi(r)-\pi(r+1)|_1\le 1$ for all $r=s,s+1,\ldots,t-1$, and it is called open in $\omega$, or $\omega$-open for short, if $\omega(r,\pi(r))=0$ for all $r=s+1,\ldots,t$. We emphasize that the initial site of an open path is not required to be open. For $A,B\subset\Z^d$ and $s< t$, we write
\begin{align*}
\{(s,A)\leftrightarrow(t,B)\}
\end{align*}
for the event that there exists an open path $\pi:\{s,s+1,...,t\}\to\Z^d$ such that $\pi(s)\in A$ and $\pi(t)\in B$, and
\begin{align*}
\{(s,A)\leftrightarrow\infty\}:=\bigcap_{t\geq s}\{(s,A)\leftrightarrow(t,\Z^d)\}.
\end{align*}
It is proved in~\cite{BG91} that there exists $\pcr\in(0,1)$ such that there exists an infinite open path from $(0,0)\in\Z_+\times \Z^d$ with positive probability if and only if the percolation parameter $p>\pcr$. In what follows, we always assume $p>\pcr$ so that $\P((0,0)\leftrightarrow \infty)>0$. Let
\begin{align*}
N_n(\omega):=\#\{\text{open paths }\pi\text{ of length $n$ starting at }(0,0)\}.
\end{align*}
For a path $\pi$ and an environment $\omega$, we define the energy by
\begin{align}\label{eq:H_n}
H_n(\omega,\pi):=\#\{t\in\{1,...,n\}:(t,\pi(t))\text{ is closed}\},
\end{align}
and for $\beta  \in[0,\infty]$, called the inverse temperature, the partition function by
\begin{align}\label{eq:partition_funct}
Z^\beta_n(\omega):=\sum_\pi e^{-\beta H_n(\omega,\pi)},
\end{align}
where $\pi$ runs over all paths of length $n$ starting in $(0,0)$. We use the convention $e^{-\infty}=0$, so that $Z_n^\infty=\sum_\pi\1_{\{\pi\text{ is open}\}}=N_n$. Note that we have $\lim_{\beta\to\infty} Z^\beta_n(\omega)=N_n(\omega)$ as long as $n\in\N$ is fixed.

For finite $\beta \ge 0$, this model is intensively studied under the name of \emph{directed polymer in random environment} and relatively well-understood. The following result is proved in~\cite[Proposition~1.5]{CSY03}. 
\begin{theorem}\label{thm:pos_temp}
For every $\beta\in[0,\infty)$ and $p\in[0,1]$, there exists $\f(\beta,p)\in (-\infty,\log(2d+1)]$ such that, $\P$-almost surely,
\begin{align}\label{eq:free}
\lim_{n\to\infty} \frac 1n\log Z^\beta_n(\omega)=\lim_{n\to\infty} \frac 1n\E[\log Z^\beta_n(\omega)]=\f(\beta,p).
\end{align}
\end{theorem}
Moreover, following the argument in~\cite[Section~2]{CFNY15}, one can shows that the function $\f\colon [0,\infty)\times (\pcr,1]\to[0,\log(2d+1)]$ is jointly continuous. In fact, much more is known including the behavior of paths under the corresponding Gibbs measure. See~\cite{Com17} for a recent detailed review. Note that the partition function \eqref{eq:partition_funct} differs from the usual definition because we do not divide by the number of paths, which results in a shift by $\log(2d+1)$ in \eqref{eq:free} compared to the usual free energy.

The model at zero temperature, i.e., $\beta=\infty$, is a singular limit of the directed polymer model and more difficult to analyze. The result most relevant to us is the existence of the free energy, which is the growth rate of the number of open paths. 
\begin{theorem}[\cite{GGM17}]
\label{thm:ggm}
For every $p>\pcr$, there exists $\tilde\alpha_p\in(0,\log(2d+1))$ such that $\P(\cdot|(0,0)\leftrightarrow\infty)$-almost surely,
\begin{align*}
\lim_{n\to\infty} \frac 1n\log N_n=\tilde\alpha_p.
\end{align*}
\end{theorem}
\begin{remark}
In the earlier works~\cite{Dar91,Yos08b}, it was proved that $\tilde\alpha_p=\log (p(2d+1))$ when $d\ge 3$ and $p$ is sufficiently large. In addition, the upper bound $\tilde\alpha(p)\leq \log(p(2d+1))$ is valid in any dimension and in the whole super-critical phase. On the other hand, it is proved in~\cite{Lac12} that this inequality can be strict. A simple argument for $\tilde\alpha_p>0$ is sketched in \cite[Remark 1.4]{GGM17}.
\end{remark}
While Theorem~\ref{thm:ggm} establishes the existence of the growth rate, it does not tell us much about $\tilde\alpha_p$ (see \emph{Remarks and open questions} in~\cite[p.4075]{GGM17}). First, it is proved in~\cite[Theorem~1.2]{GGM17} that the growth rate along any prescribed slope $x$ exists and that $\tilde\alpha_p$ is realized as the supremum over $x$. Though the supremum is attained at $x=0$, it is not proved to be a unique maximizer. This is related to a major open problem in the directed polymer context~\cite[Open Problem 9.3]{Com17} and seems to be a rather hard problem. Second, it is not known if $\tilde\alpha_p$ is continuous in $p$. One of the results in this paper proves this continuity. 

In view of the aforementioned $\lim_{\beta\to\infty} Z^\beta_n(\omega)=N_n(\omega)$, we will sometimes use $\f(\infty,p)$ in place of  $\tilde\alpha_p$ in the sequel. This in particular makes the statement of Theorem~\ref{thm:main} intuitive. 

\section{Main results}
The following two theorems are about the rate of convergence. The first shows that the finite volume free energy $\frac1n \log Z_n^\beta$ is tightly concentrated around its conditional mean. The second shows that the conditional mean is close to the free energy. 
\begin{theorem}
\label{thm:conc}
For any $\delta,\eps>0$ and $r>0$, there exists $c_{\ref{thm:conc}}>0$ such that for all $\beta\in[0,\infty]$, $p> \pcr+\eps$ and all $n\in\N$,
\begin{equation}
  \label{eq:conc}
\P\left(\left|\log Z_n^\beta-\E[\log Z_n^\beta\mid (0,0)\leftrightarrow\infty]\right| \ge n^{\frac12+\delta}\cond (0,0)\leftrightarrow\infty \right)
\le c_{\ref{thm:conc}} n^{-r}.
\end{equation}
\end{theorem}
\begin{theorem}
\label{thm:NRF}
For any $\delta,\eps>0$, there exists $c_{\ref{thm:NRF}}>0$ such that for all $\beta\in[0,\infty]$, $p> \pcr+\eps$ and all $n\in\N$,
\begin{equation}
\left|\frac1n \E[\log Z_n^\beta\mid (0,0)\leftrightarrow\infty]-\f(\beta,p)\right| \le c_{\ref{thm:NRF}}n^{-\frac12+\delta}.
\label{eq:NF} 
\end{equation}
\end{theorem}
It is important that we have these results uniformly in $\beta\in[0,\infty]$ and $p$ away from criticality. We need the conditioning in Theorem~\ref{thm:conc} not only at $\beta=\infty$, where its necessity is obvious, but also for $\beta<\infty$ to make the bound uniform. Indeed, since $\log Z_n^\beta \ge 0$ on $\{(0,0)\leftrightarrow \infty\}$ and $\log Z_n^\beta \le -\beta + n\log (2d+1)$ on $\{(0,0)\not\leftrightarrow \infty\}$, a concentration around a deterministic value like~\eqref{eq:conc} cannot hold for $\beta\ge 2n\log (2d+1)$. 

Thanks to the above uniform bounds, we can easily get the following continuity results. 
\begin{theorem}\label{thm:main}
For every $p>\pcr$,
\begin{align}
\lim_{\beta\to\infty} \f(\beta,p)&=\f(\infty,p)\label{eq:zero_temperature_limit},\\
\lim_{q\to p}\f(\infty,q)&=\f(\infty,p).\label{eq:cont_in_p}
\end{align}
\end{theorem} 

\begin{proof}[Proof of Theorem \ref{thm:main} assuming Theorem~\ref{thm:NRF}]
For any $\epsilon>0$ and $p_0\in (\pcr,p)$, we can choose $n\in\N$ such that 
\begin{equation}
  \label{eq:uniformpbeta}
\sup_{(p,\beta)\in [p_0,1]\times [0,\infty]}\left|\frac1n \E[\log Z_n^\beta\mid (0,0)\leftrightarrow\infty]-\f(\beta,p)\right| \le \epsilon
\end{equation}
by Theorem~\ref{thm:NRF}. 

Let us prove~\eqref{eq:zero_temperature_limit} first. Since $n$ is fixed, we can choose $\beta>0$ so large that 
\begin{equation}
\left|\frac1n \E[\log Z_n^\beta\mid (0,0)\leftrightarrow\infty]-\frac1n \E[\log Z_n^\infty\mid (0,0)\leftrightarrow\infty]\right|\le \epsilon.
\end{equation}
Then it follows that
\begin{equation}
\begin{split}
 |\f(\beta,p)-\f(\infty,p)| 
&\le  \left|\frac1n \E[\log Z_n^\beta\mid (0,0)\leftrightarrow\infty]-\f(\beta,p)\right|\\
&\quad +\left|\frac1n \E[\log Z_n^\beta\mid (0,0)\leftrightarrow\infty]-\frac1n \E[\log Z_n^\infty\mid (0,0)\leftrightarrow\infty]\right|\\
&\quad +\left|\frac1n \E[\log Z_n^\infty\mid (0,0)\leftrightarrow\infty]-\f(\infty,p)\right|\\
&\le  3\epsilon.
\end{split}
\end{equation}

To prove \eqref{eq:cont_in_p}, we use independent $\text{Unif}[0,1]$ random variables $((U(t,x))_{(t,x)\in\Z_+\times\Z^d},\mathbb{P})$ and define $\omega_p(t,x)=1_{\{U(t,x) \le p\}}$. 
Noting the fact that 
\begin{align*}
  0\le \frac{1}{n}\log Z_n^\infty \le \frac{1}{n}\log Z_n^0 =\log(2d+1)\text{ on }\{(0,0)\leftrightarrow\infty\},
\end{align*}
we have for $q>p \geq p_0$, 
\begin{align*}
&\left|\frac1n \E[\log Z_n^\infty, (0,0)\leftrightarrow\infty]-\frac1n \Eq[\log Z_n^\infty, (0,0)\leftrightarrow\infty]\right|\\
&\quad \le \mathbb{E}\left[\frac{1}{n} \log Z_n^\infty(\omega_q) 1_{\{(0,0)\leftrightarrow\infty\}}(\omega_q)-\frac{1}{n} \log Z_n^\infty(\omega_p)1_{\{(0,0)\leftrightarrow\infty\}}(\omega_p)\right]\\
&\quad \le \mathbb{E}\left[\left(\frac{1}{n} \log Z_n^\infty(\omega_q)-\frac{1}{n} \log Z_n^\infty(\omega_p)\right)1_{\{(0,0)\leftrightarrow\infty\}}(\omega_p)\right]\\
&\qquad +\log(2d+1)\left(\Pq((0,0)\leftrightarrow\infty)-\P((0,0)\leftrightarrow\infty)\right).
\end{align*}
Since $\lim_{|p-q|\to 0}\mathbb{P}(\bigcap_{0\le t\le n, |x|_1\le n}\{\omega_p(t,x)=\omega_q(t,x)\})=1$, the first term on the right-hand side converges to zero as $|p-q|\to 0$. The second term also converges to zero as $|p-q|\to 0$, thanks to the continuity of the percolation probability proved in~\cite[Theorem~2]{GH02}. Using the continuity once more, we conclude that  
\begin{align*}
  \lim_{|p-q|\to 0}\left|\frac1n \E[\log Z_n^\infty\mid (0,0)\leftrightarrow\infty]-\frac1n \Eq[\log Z_n^\infty\mid (0,0)\leftrightarrow\infty]\right|=0.
\end{align*}
From this and~\eqref{eq:uniformpbeta}, we obtain~\eqref{eq:cont_in_p}. 
\end{proof}
We close this section with a few remarks.
\begin{remark}
The results in this section are formulated under the conditioning on $\{(0,0)\leftrightarrow \infty\}$. This is for the ease of notation and also to make the correspondence to Theorem~\ref{thm:ggm} clear. However, as is usual for percolation models, this is essentially the same as conditionning on $\{(0,0)\leftrightarrow (n,\Z^d)\}$. See Theorem~\ref{thm:largefinite}. 
\end{remark}
\begin{remark}
Although we use the result of~\cite{GGM17} to ensure the existence of $\f(\infty,p)$, it can also be deduced from our argument. See Remark~\ref{rem:existence}. 
\end{remark}
\begin{remark}
There is a recent work~\cite{DCKNPS} which shows that the number of maximal paths in directed last-passage percolation grows exponentially even in the sub-critical phase. There is an interesting fact revealed by their argument: the growth rate $\tilde{\alpha}(p)$ does not tends to 0 as $p\searrow \pcr$. One might wonder if the almost sure growth rate exists in the sub-critical phase, but it seems to be a hard problem. Both of the arguments in~\cite{GGM17} and in this paper rely on the sub-additivity of the number of open paths that fails to hold for maximal paths. 
\end{remark}

\section{Technical difficulty and new idea}
\label{sec:difficulty}
We think it is appropriate to explain the technical difficulty in the proof since the rates of convergence results like Theorems~\ref{thm:conc} and~\ref{thm:NRF} are standard for the directed polymer models in positive temperature, i.e., $\beta<\infty$. The concentration around the mean for the directed polymer with $\beta<\infty$ is proved in~\cite[Theorem~1.4]{LW09} under a mild assumption on the random potential $\omega$. Once the concentration around the mean is established, the non-random fluctuation bound follows rather easily by using the argument in~\cite{Z10}. At zero temperature, it is less standard but still there are some recent progress~\cite{N18,FJ19}. 

We emphasize that in all these works, the concentration around the mean is proved by using appropriate variants of the so-called \emph{bounded difference inequality}. This requires to control the influence caused by changing the environment at one time slice. It is in the control of the influence where we have a difficulty in our model. 

More precisely, when we change the environment at one time slice, it is possible that most of the open paths get disconnected. In this situation each time-slice would have a large influence, so we have to exclude this possibility to get a useful concentration. This behavior might look very unlikely since the percolation cluster grows like a cone and is thus supported on a large number of sites. However, as is proved in~\cite{Yos10}, the spatial distribution of the open paths can localize and then it is possible to disconnect many open paths by closing relatively small number of sites. Just to give an intuition, we recall that for the directed polymer in the Gaussian environment, it is proved in~\cite[Theorem~6.1]{Com17} that the distribution of the paths tends to concentrate in a small neighborhood of a ``favorite path'' as $\beta\to\infty$. 

There was a similar problem in~\cite{FJ19} where we established the rate of convergence results for the Brownian directed polymer in Poissonian disasters. We circumvented it by using a variant of the bounded difference inequality~\cite[Theorem~15.5]{BLM13} which can also be viewed as an extension of the Efron--Stein inequality. The advantage of the variant is that, roughly speaking, it only requires a moment bound for the influence, instead of the $L^\infty$-bound. We managed to get such a bound by constructing a repairing path as shown in Figure~\ref{fig:1}(i). More precisely, we showed that the defect caused by resampling a time-slice can be fixed using a repairing path of bounded length. 
This part of the argument in~\cite{FJ19} crucially uses the fact that the Brownian motion can move an arbitrarily large distance in a unit time. 

In the current nearest-neighbor setting, it gets difficult to find a short repairing path as above. We have essentially no information about the spatial distribution of the open paths and in particular, we cannot preclude the following possibility (see Figure~\ref{fig:1}(ii)):
\begin{itemize}
 \item the open paths are concentrated around a favorite path as in~\cite[Theorem~6.1]{Com17} and
 \item the favorite path locally follows an edge of the percolation cone. 
\end{itemize} 
\begin{figure}[b]
\includegraphics[width=.49\textwidth]{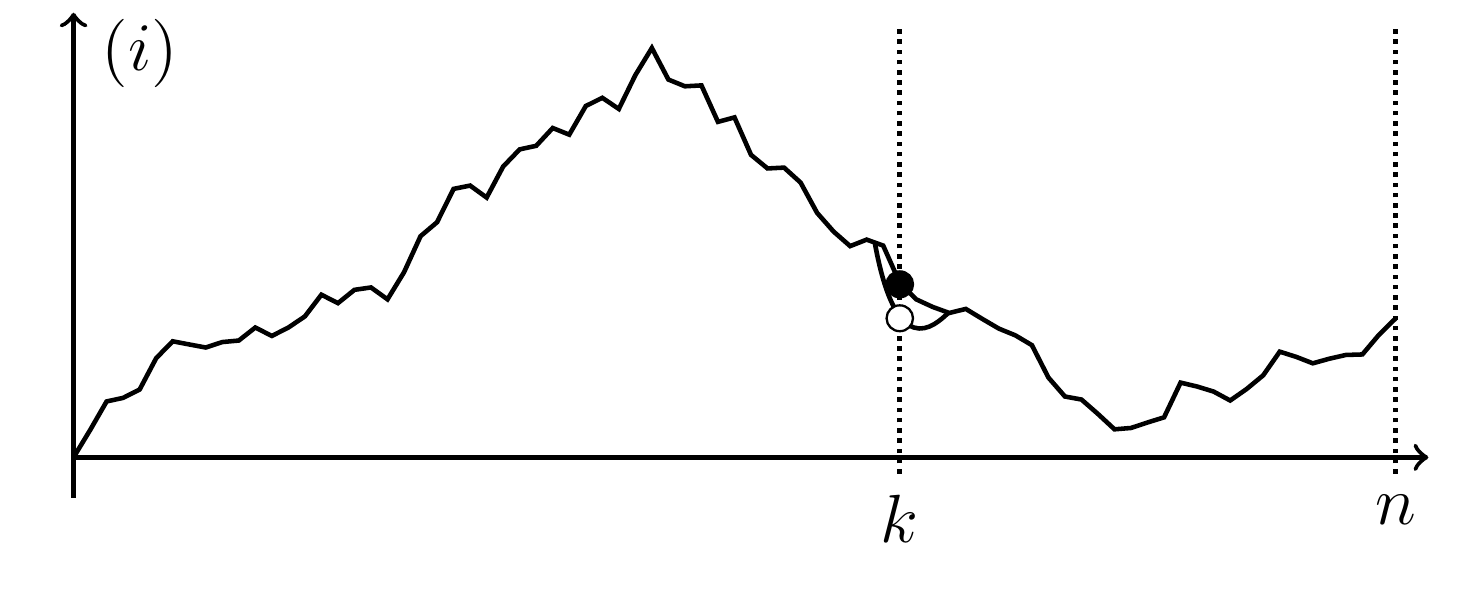}
\includegraphics[width=.49\textwidth]{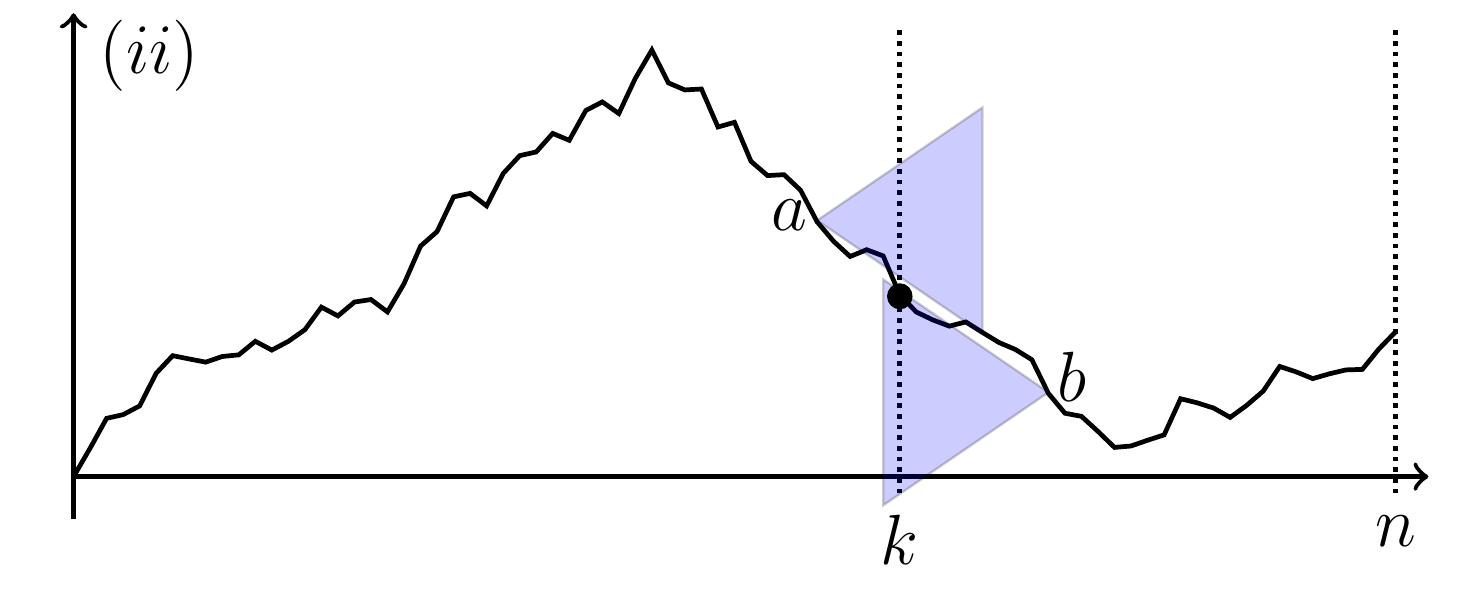}
\includegraphics[width=.49\textwidth]{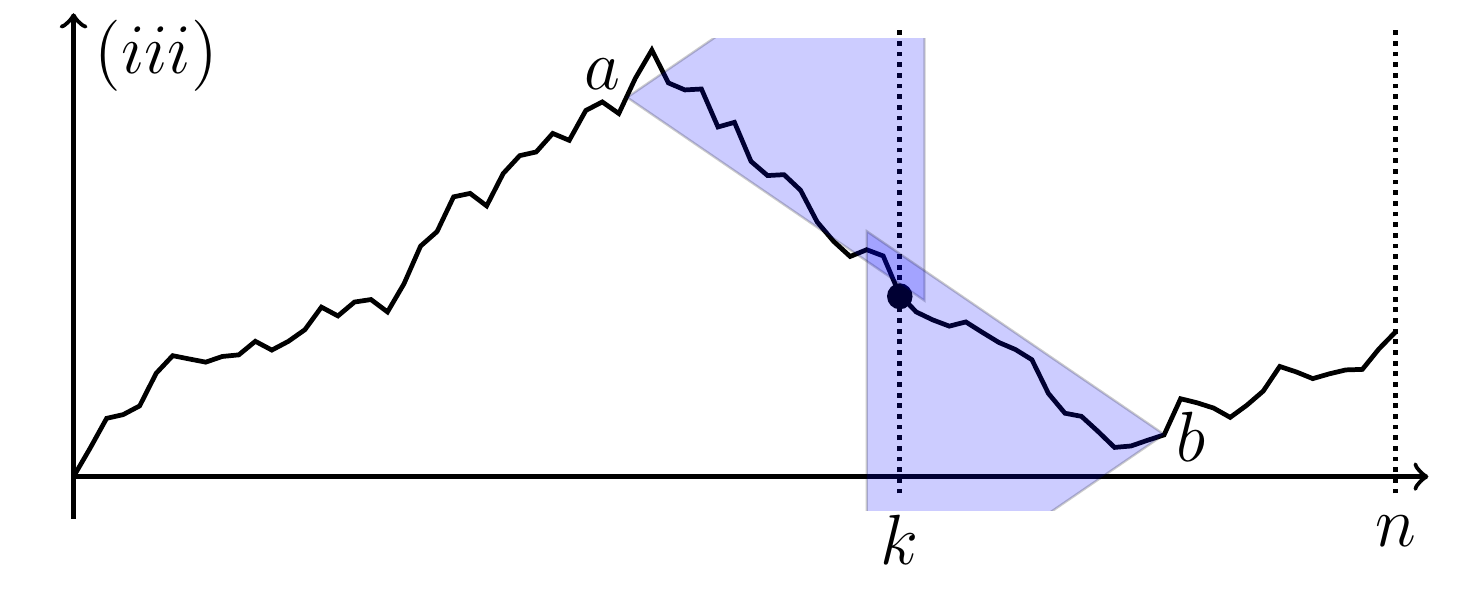}
\includegraphics[width=.49\textwidth]{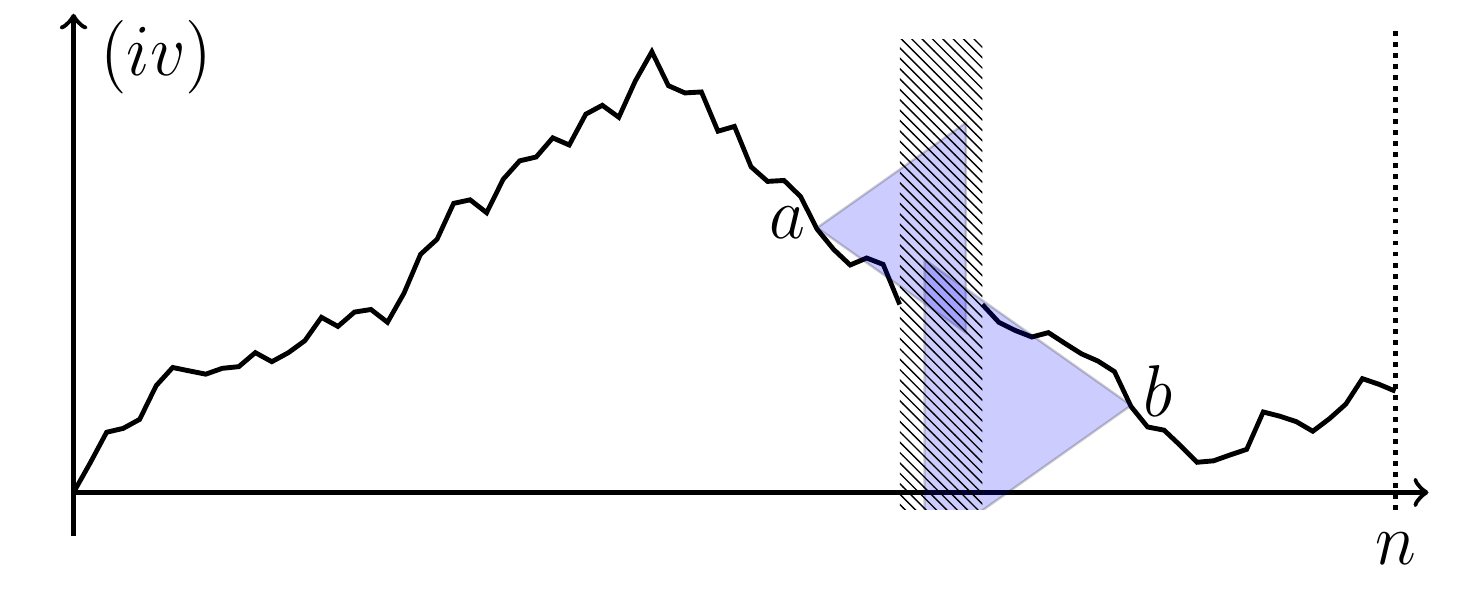}
\caption{Top left: The repairing path in the case of Brownian directed polymer. After moving a constant distance away from the defect (filled circle), one can show that switching to the repairing path (empty circle) does not cost too much. Top right: The bad situation for the nearest-neighbor model. The path has a large slope and the percolation cones (shaded) from two points in the future and past may miss each other. Bottom left: In the bad situation, we can still find a repairing path by considering percolation points far enough in the past and future. But then the influence to the number of paths becomes large. Bottom right: Our new idea is to insert an independent slab (striped) at time $k$, which gives the percolation cones more time to meet.} 
\label{fig:1}
\end{figure}
Suppose that the sites around the 
favorite path is closed after resampling the environment at time $k$ and that it has an extremal slope around time $k$. The shape theorem for oriented percolation shows that there are no open paths with greater slope. We can still find an open path that ``reconnects'' the favorite path in the resampled environment by moving far away from time $k$, but this repairing path will be very long. In this situation, many paths in the original environment would be mapped to the same repairing path in the new environment, and this makes the bounded difference inequality inefficient.

To explain the new idea, let us first recall that the bounded difference inequality is a special case of the Azuma--Hoeffding inequality: the latter requires a bound on the martingale difference
\begin{equation*}
\Delta M_k=\E[ \1_{\{(0,0)\leftrightarrow (n,\Z^d)\}}\log N_n \mid \F_k]-\E[ \1_{\{(0,0)\leftrightarrow (n,\Z^d)\}}\log N_n \mid \F_{k-1}].
\end{equation*}
and the former suggests to get such a bound by considering a special coupling between the two random variables that uses the same environment except at time $k$. Our new idea is to use a different coupling.

By independence, the above conditional expectations are simply an integration of the environment after time $k+1$, resp.\ after time $k$. Because of this, it is enough to bound the difference between $\log N_n(\omega_1)$ and $\log N_n(\omega_2)$, where $\omega_2$ is obtained by inserting an extra ``slab'' with an independent environment into $\omega_1$ at time $k$, see Figure~\ref{fig:1}(iv). At first glance, this does not seem to be an improvement because an $\omega_1$-path can still be closed in $\omega_2$, but the additional slab is very helpful to construct a short repairing path. Indeed, with high probability we can find forward and backward percolation points on the path before and after---but not too far away from---time $k$. Thanks to the inserted slab, the forward and backward percolation cones starting from those points meet with high probability, which results in an open path connecting those two points. The crucial point is that the slab allows additional time for the percolation cones to meet, even if the original path has very high slope, whereas the cones might miss each other in the standard coupling used in the bounded difference inequality. We make this argument rigorous by using some ideas from~\cite{GGM17}. 

We stress that to obtain a uniform concentration inequality, it is necessary to use the above construction even in the case $\beta<\infty$ because the influence bounds obtained by standard methods diverge for $\beta\to\infty$. In fact, this introduces an additional complication since we need to find repairing paths even if the original path is not open and potentially visits regions with few open sites.

Finally, let us make a comment on the concentration inequality we will use to prove Theorem~\ref{thm:conc}. As mentioned above, we can find a repairing path only with high probability. Thus the Azuma--Hoeffding inequality in not suited for our purpose. In the previous work~\cite{FJ19}, we used the variant~\cite[Theorem~15.5]{BLM13}, but it uses the same coupling as the bounded difference inequality and hence does not allow us to insert the slab. For these reasons, we ended up using Burkholder's classical inequality in~\cite{B66} for the moments of martingales.

\section{Toolbox}
\label{sec:toolbox}
In order to find the repairing path mentioned in Section~\ref{sec:difficulty}, we need some results on the oriented percolation. In this section, we collect known results and introduce some definitions. In the following, $\ell_n$ denotes a small auxiliary time-scale 
\begin{align}\label{eq:ell}
\ell_n:=\lfloor (\log n)^2 \rfloor.
\end{align}
To keep the notation compact, we will often denote sets of vertices by $A\subseteq\R^d$ and time intervals by $I\subseteq\R_+$ with the understanding that we refer to $A\cap\Z^d$ and $I\cap\Z_+$.

\subsection{Known results}
\label{sec:known}
We recall some results from the literature. Though all the results hold uniformly in $p>\pcr+\epsilon$, it is usually not made precise and sometimes not entirely obvious from the proofs. In the appendix, we give references and some comments on how to ensure that constants are uniform in $p>\pcr+\epsilon$. First, long finite clusters are unlikely.
\begin{theorem}
\label{thm:largefinite}
For every $\eps>0$ there exists $c_{\ref{thm:largefinite}}>0$ such that for all $p> \pcr+\eps$ and all $n\in\N$,
\begin{align*}
\P\left((0,0)\leftrightarrow(n,\Z^d),(0,0)\nleftrightarrow\infty\right)\leq e^{-c_{\ref{thm:largefinite}}n}.
\end{align*}
\end{theorem}

In the proofs below, we will frequently encounter events where, say, $(0,0)\leftrightarrow(n,\Z^d)$ and $(0,\Z^d)\leftrightarrow (n,0)$, from which we want to conclude that $(0,0)\leftrightarrow(n,0)$. In dimension $d=1$ it is not too hard to make that conclusion rigorous using that the left- and rightmost paths in oriented percolation have an asymptotic slope \cite[Section 3]{D84} and a natural path-crossing argument. 

In higher dimensions, the same argument does not work because the forward percolation cluster from $(0,0)$ and the backward percolation cluster from $(n,0)$ could, in principle, miss each other. We therefore recall the so-called coupled zone, which is a subset of the percolation cluster where this kind a behavior does not occur. We need that it grows in a linear speed. 

\begin{definition}\label{def:coupled}
Let $v>0$ be a small number to be chosen in Section~\ref{sec:5.3} and set, for non-negative integers $k$ and $n$,
\begin{align}\label{eq:def_coupled}
\C^{n,x}_k&:=\left\{\omega:
\begin{array}{l}
(n,x)\leftrightarrow (n+k,\Z^d) \text{ and }\text{for every }y\in x+[-kv,kv]^d,\\
\text{either }(n,\Z^d)\nleftrightarrow (n+k,y)\text{ or }(n,x)\leftrightarrow(n+k,y)
\end{array}\right\},\\
\accentset\leftarrow\C^{n,x}_{k}&:=\left\{\omega:
\begin{array}{l}
(n-k,\Z^d)\leftrightarrow (n,x) \text{ and }\text{for every }y\in x+[-kv, kv]^d,\\
\text{either }(n-k,y)\nleftrightarrow (n,\Z^d)\text{ or }(n-k,y)\leftrightarrow(n,x)
\end{array}\right\},\label{eq:def_coupled_rev}
\end{align}
with the convention $C_k:=C^{0,0}_k$.
\end{definition}

\begin{remark}
  The $\leftarrow$ in~\eqref{eq:def_coupled_rev} indicates that it is an event concerning the backward percolation. We stress that even when we consider the backward percolation, we keep the convention on the open path: the site with the smallest time index may be closed while the one with the largest time index has to be open. 
\end{remark}

\begin{theorem}\label{thm:coupled}
For any $\epsilon>0$, there exists $c_{\ref{thm:coupled}}>0$ such that for all $p>\pcr+\eps$ and all $k\geq 0$,
\begin{align*}
\P\left(\{(0,0)\leftrightarrow(k,\Z^d)\} \setminus C_k\right)\leq e^{-c_{\ref{thm:coupled}} k}.
\end{align*}
\end{theorem}

Finally, we need an estimate on the percolation probability as a function of the initial set: 
\begin{theorem}\label{thm:large_initial}
For every $\eps>0$, there exists $c_{\ref{thm:large_initial}}$ such that for all $p>\pcr+\eps$ and all finite $A\subset\Z^d$,
\begin{align}\label{eq:large_initial}
\P\left((0,A)\nleftrightarrow\infty\right)\leq e^{-c_{\ref{thm:large_initial}}|A|}.
\end{align}
\end{theorem}

\subsection{Good events}
\label{sec:good}
Our construction of the repairing path goes roughly as follows: Referring to Figure~\ref{fig:1}(iv), with high probability, 
\begin{enumerate}
 \item[(1)] there are forward and backward percolation points in distance of order $\ell_n^2$ from $k$ by Theorem~\ref{thm:largefinite}, denoted by $a$ and $b$ in Figure~\ref{fig:1}(iv),
 \item[(2)] there are many points in $\{k-1\}\times\Z^d$ connected to $a$, 
 \item[(3)] one of the points is followed by an infinite open path by Theorem~\ref{thm:large_initial},
 \item[(4)] the coupled zone from that point grows linearly by Theorem~\ref{thm:coupled},
 \item[(5)] similarly to (2) and (3), the point $b$ has an infinite backward open path after the slab inserted, 
 \item[(6)] there exists an open path connecting $a$ and $b$ by Definition~\ref{def:coupled}.
\end{enumerate}
For step (2), we thought that estimates on the size of the infected region
\begin{align*}
|\{y\in\Z^d:(0,0)\leftrightarrow(y,n)\}|
\end{align*}
should be known but we were not able to find it in the literature. It is possible to establish this type of result directly but we instead introduce abstract events to indicate that the configuration before and after time $k$ are well-behaved. We found it easier to prove the probability bound for these events, and also the abstract formulation simplifies the later arguments.

\begin{definition}\label{def:good}
Let $c_{\ref{def:good}}:=c_{\ref{thm:largefinite}}/5$ and set
\begin{align}
\G_k^{n,x}&:=\left\{\omega:\P\left((n,x)\leftrightarrow \infty\cond\mathcal \F_{[n,n+k]}
\right)\geq 1- e^{-c_{\cref{def:good}}k}\right\},\label{eq:def_initial}\\
\GG_k^{n,x}&:=\left\{\omega:\P\left(-\infty\leftrightarrow (n,x)\cond \F_{[n-k,n]}
\right)\geq 1- e^{-c_{\cref{def:good}}k}\right\}
\end{align}
with the convention $\G_k:=\G_k^{0,0}$ and $\G_{k,k+1}^{n,x}:=\G_k^{n,x}\cap \G_{k+1}^{n,x}$. Moreover, let
\begin{align}\label{eq:def_very_good}
\GGG_k:=\left\{\omega:\P\left((0,0)\leftrightarrow \infty\cond\mathcal \F_{k}\right)\geq 1- e^{-4c_{\cref{def:good}}k}\right\}.
\end{align}
\end{definition}
We think of $\omega\in\G$ as a \emph{good} environment and $\omega\in\GGG$ a \emph{very good} environment. The second condition $\GGG$ is necessary for technical reasons and is used only in Lemmas \ref{lem:AisTypical} and \ref{lem:A'isTypical}. It is easy to see that $G_k$ has high probability if the origin percolates.
\begin{lemma}\label{lem:good}
For any $\epsilon>0$ and $p>\pcr+\eps$, the following hold:
\begin{enumerate}
 \item[(i)]
For every $k\geq 0$,
\begin{align}
\P\left(\{(0,0)\leftrightarrow(k,\Z^d)\}\setminus \G_k\right)\leq \P\left(\{(0,0)\leftrightarrow(k,\Z^d)\}\setminus \GGG_k\right)\leq e^{-c_{\ref{def:good}}k}\label{eq:first}.
\end{align}
\item[(ii)] For every $0\leq k\leq l\leq 3k$, almost surely on $\GGG_k$,
\begin{align}
\P\left(\G_l^c\cond\F_k\right)\leq e^{-c_{\ref{def:good}}k}.\label{eq:second}
\end{align}
\end{enumerate}
\end{lemma}
\begin{proof}
Since $\GGG_k\subseteq\G_k$, we only have to prove the second inequality. From Theorem \ref{thm:largefinite} and the definition of $\GGG_k$, we have
\begin{align*}
e^{-c_{\ref{thm:largefinite}}k}
&\geq \P\left((0,0)\leftrightarrow(k,\Z^d),(0,0)\nleftrightarrow\infty\right)\\
&\geq \E\left[\P((0,0)\nleftrightarrow\infty|\F_k)\1_{\{(0,0)\leftrightarrow(k,\Z^d)\}\setminus \GGG_k}\right]\\
&\geq e^{-4 c_{\ref{def:good}}k} \P\left(\{(0,0)\leftrightarrow(k,\Z^d)\}\setminus \GGG_k\right).
\end{align*}
Part (i) follows from our choice $c_{\ref{def:good}}=c_{\ref{thm:largefinite}}/5$. 
For part (ii) we have, on $\GGG_k$,
\begin{align*}
e^{-\frac 45 c_{\ref{thm:largefinite}}k}&\geq \P((0,0)\nleftrightarrow\infty|\F_k)\\
&\geq \E\left[\P((0,0)\nleftrightarrow\infty|\F_l)\1_{\G_l^c}\cond\F_k\right]\\
&\geq e^{-\frac 15 c_{\ref{thm:largefinite}} l}\P\left(\G_l^c\cond\F_k\right).
\end{align*}
For $l\leq 3k$, we indeed get $\P\left(\G_l^c\cond\F_k\right)\leq e^{-\frac 15c_{\ref{thm:largefinite}}k}= e^{-c_{\ref{def:good}} k}$.
\end{proof}

\subsection{Well-connected environments}

As explained in Section~\ref{sec:difficulty}, one of the main difficulties is to ensure that whenever a path of length $n$ is disconnected at level $k$, we can find a reasonably short repairing path around the the defect. As a first step in this direction, we introduce an event $\A_{n,k}$ which could be called the \emph{anticipating connection} event, and its backward version $\AA_{n,k}$. 

\begin{definition}\label{def:connected}
Let $\ell_n^2\leq k\leq n$ and 
\begin{align}
\A_{n,k}&:=\left\{\omega :
\begin{array}{l}
\text{For every path }\pi\text{ connecting }(0,0)\leftrightarrow(k,\Z^d),\\
\text{there exists } j\in [\ell_n^2,2\ell_n^2\wedge k]\text{ such that }\omega\in\G_{j,j+1}^{k-j,\pi(k-j)}
\end{array}\right\}, 
\label{eq:def_A}\\
\AA_{n,k}&:=\bigcap_{x\in[-k-\ell_n^2,k+\ell_n^2]^d}\left\{\omega: \text{Either }\omega\in\GG^{k+\ell_n^2,x}_{\ell_n^2-1}\text{ or }(k+1,\Z^d)\nleftrightarrow(k+\ell_n^2,x)\right\}.\label{eq:def_AA}
\end{align}
\end{definition}

Following the explanation before Definition~\ref{def:good}, one may interpret the event $\A_{n,k}$ as ``for every open path $\pi$, there is an earlier point $(k-j,\pi(k-j))$ from which there are many connections to $(k,\Z^d)$ and $(k+1,\Z^d)$''. 

\section{Concentration inequality}\label{sec:conc}

We will first prove Theorem~\ref{thm:conc} at zero temperature in Section~\ref{sec:conc_zero}, see Proposition \ref{prop:conc_zero}. The most important idea appears in this case and we think it is better to present it in the simplest setting first. The uniform concentration bound in positive temperature will be proved in Section~\ref{sec:conc_positive} by adapting the argument at zero temperature. 

\subsection{Concentration inequality at zero temperature}
\label{sec:conc_zero}

The goal of this section is to prove Theorem~\ref{thm:conc} at zero temperature. Throughout this section, we use the simple and suggestive notation $N_n$ instead of $Z_n^\infty$. The technical core is the proof of the following proposition.
\begin{proposition}\label{prop:conc_zero}
For every $\epsilon, \delta>0$ and $r>0$, there exist $c_{\cref{prop:conc_zero}}>0$ such that for all $n\in\N$ and all $p>\pcr+\eps$,
\begin{align}
\P\left((0,0)\leftrightarrow (n,\Z^d), \left|\log N_n-\E[\log N_n\mid (0,0)\leftrightarrow (n,\Z^d)]\right|\ge n^{\frac12+\delta}\right)\le c_{\cref{prop:conc_zero}}n^{-r}.
\end{align}
\end{proposition}
Once this proposition is proved, it is easy to see that one can change $(0,0)\leftrightarrow (n,\Z^d)$ to $(0,0)\leftrightarrow\infty$ by using Theorem~\ref{thm:largefinite} and $0\le (\log N_n)1_{\{(0,0)\leftrightarrow (n,\Z^d)\}} \le n\log(2d+1)$. 
Then, since the probability of $\P((0,0)\leftrightarrow \infty)$ is bounded away from zero for $p>\pcr+\eps$, the concentration under the conditional probability in \eqref{eq:conc} follows. 

In order for the conditional expectations to make sense in the case where the origin does not percolate, we define 
\begin{align*}
\widetilde N_n:=N_n+\1_{\{(0,0)\nleftrightarrow (n,\Z^d)\}},
\end{align*}
so that $\log \widetilde N_n=\1_{\{(0,0)\leftrightarrow (n,\Z^d)\}}\log N_n$. 
Our proof of Proposition~\ref{prop:conc_zero} consists of two steps. We first prove a concentration conditionally on $\F_{\ell_n^2}$ for environments whose initial part is \emph{very good}, i.e., $\omega\in \GGG_{\ell_n^2}$. 
\begin{lemma}
\label{lem:conc_1}
For every $\epsilon, \delta>0$ and $r>0$, there exist $c_{\cref{lem:conc_1}}>0$ such that for all $n\in\N$, $p>\pcr+\eps$ and $\omega\in \GGG_{\ell_n^2}$, 
\begin{equation}
\P\left(\left|\log\widetilde N_n-\E[\log\widetilde N_n\mid \F_{\ell_n^2}]\right|\ge \tfrac12 n^{\frac12+\delta}\cond \F_{\ell_n^2}\right)(\omega)\le c_{\cref{lem:conc_1}}n^{-r}.
\end{equation}
\end{lemma}
Then in the second step, we get rid of the conditioning by showing that the contribution from the remaining environments is negligible. 
\begin{lemma}
\label{lem:conc_2}
For every $\epsilon,\delta>0$, there exists $c_{\cref{lem:conc_2}}>0$ such that for all $n\in\N$, $p>\pcr+\eps$ and $\omega\in \GGG_{\ell_n^2}$, 
\begin{equation}
\left|\E[\log\widetilde N_n\mid \F_{\ell_n^2}](\omega)-\E[\log \widetilde N_n\mid (0,0)\leftrightarrow (n,\Z^d)]\right| < c_{\cref{lem:conc_2}}n^{\frac12+\delta}.
\end{equation}
\end{lemma}
\begin{proof}
[Proof of Proposition~\ref{prop:conc_zero} assuming Lemmas~\ref{lem:conc_1} and~\ref{lem:conc_2}] 
Noting that $N_n=\widetilde N_n$ on $\{(0,0)\leftrightarrow (n,\Z^d)\}$, we have
\begin{equation}
\begin{split}
&\P\left((0,0)\leftrightarrow (n,\Z^d),\left|\log N_n-\E[\log N_n\mid (0,0)\leftrightarrow (n,\Z^d)]\right|\ge n^{\frac12+\delta}\right)\\
&\quad \le \P\left(\{(0,0)\leftrightarrow (n,\Z^d)\}\setminus \GGG_{\ell_n^2} \right)\\
&\qquad +\P\left(\GGG_{\ell_n^2},\left|\log \widetilde N_n-\E[\log \widetilde N_n\mid (0,0)\leftrightarrow (n,\Z^d)]\right|\ge n^{\frac12+\delta}\right).
\end{split}
\end{equation}
By Lemmas~\ref{lem:good}(i) and~\ref{lem:conc_2}, for all $n\ge c_{\cref{lem:conc_2}}$, the right-hand side is bounded by
\begin{equation}
\begin{split}
&e^{-{c_{\cref{def:good}}}\ell_n^2}+\P\left(\GGG_{\ell_n^2},\left|\log \widetilde N_n-\E[\log \widetilde N_n\mid \F_{\ell_n^2}]\right|\ge \tfrac12 n^{\frac12+\delta}\right)\\
&\le e^{-c_{\cref{def:good}}\ell_n^2}+\E\left[\P\left(\left|\log\widetilde N_n-\E[\log\widetilde N_n\mid \F_{\ell_n^2}]\right|\ge \tfrac12 n^{\frac12+\delta}\cond \F_{\ell_n^2}\right); \GGG_{\ell_n^2} \right].
\end{split}
\end{equation}
Finally, Lemma~\ref{lem:conc_1} shows that the last line is bounded by $c n^{-r}$.
\end{proof}

\subsubsection{Conditional concentration at zero temperature}
In this section, we prove Lemma~\ref{lem:conc_1} by applying the following moment bound due to Burkholder~\cite[Theorem~9]{B66}:
\begin{theorem}
\label{thm:Burkholder}
For every $q\in\N$, there exists $c_{\cref{thm:Burkholder}}>0$ such that for every martingale $((M_n)_{n\in\N},\mathbb{P})$ and $l<n$,
\begin{align}
\label{eq:Burkholder}
\mathbb{E}\big[(M_n-M_l)^{2q}\big]\leq c_{\cref{thm:Burkholder}}\mathbb{E}\left[\left(\sum_{k=l+1}^{n}(M_{k}-M_{k-1})^2\right)^q\right].
\end{align}
\end{theorem}
We consider the martingale
\begin{equation}\label{eq:M_n}
M_k:=\E[\log\widetilde N_n\mid \F_k],\quad \ell_n^2\le k\le n
\end{equation}
under $\P(\cdot \mid \F_{\ell_n^2})$ and therefore need a good bound on the martingale difference
\begin{align*}
\Delta_k:=\E[\log\widetilde N_n|\mathcal F_k]-\E[\log\widetilde N_n|\mathcal F_{k-1}]
\end{align*}
for $\ell_n^2+1 \le k \le n$. Observe that we have the trivial bound
\begin{align}
\label{eq:trivial}
|\Delta_k|\leq n\log(2d+1)
\end{align}
for all $\omega$, $n\in\N$ and $k\le n$. Thus it is enough to obtain a bound of the form $|\Delta_k|\le \ell_n^c$ on an event with probability larger than $1-o(n^{-q})$, and this is exactly what the two key lemmas in this section provide. Recall the events introduced in Definition \ref{def:connected}. 

\begin{lemma}
\label{lem:mart_diff}
There exists $c_{\ref{lem:mart_diff}}>0$ such that for every $n\in\N$, $\ell_n^2+1\leq k\leq n$ and $\omega \in \A_{n,k-1}$, 
\begin{equation}
\label{eq:Delta<}
\Delta_k \le c_{\ref{lem:mart_diff}} \ell_n^4\log (2d+1). 
\end{equation}
\end{lemma}

In view of the explanation following Definition~\ref{def:connected}, the assumption $\omega \in \A_{n,k-1}$ corresponds to steps (1)--(2) in Section~\ref{sec:good}, and the conclusion corresponds to steps (3)--(6). For step (5), we also need $\omega\in \AA_{n,k}$ but this appears only in the proof. 

Before giving the proof of  Lemma \ref{lem:mart_diff}, we introduce some notation that will be used throughout the rest of paper. Let $(\omegas,\Ps)$ and $(\omegae,\Pe)$ be independent copies of $(\omega,\P)$, where the superscirpts ``b'' and ``e'' indicate ``beginning'' and ``end'' respectively, and define the environment
\begin{equation}\label{eq:} 
 [\omegas,\omegae]_k(m,x)=
\begin{cases}
\omegas(m,x), &m\le k,\\
\omegae(m-k,x), &m> k.
\end{cases}
\end{equation}
Note that $[\omegas,\omegae]_k$ has the same law as the original environment. This notation allows us to write 
\begin{equation}\label{eq:notation}
 \E[\log \widetilde{N}_n\mid \F_k]=\Ee[\log \widetilde{N}_n([\omegas,\omegae]_k)].
\end{equation}
As mentioned in Section~\ref{sec:difficulty}, we will also use the following environment: 
\begin{equation}\label{eq:def_coupling2}
 [\omegas,\omegasl,\omegae]_{k,l}(m,x)=
\begin{cases}
\omegas(m,x), &m\le k,\\
\omegasl(m-k,x), &k<m\le l,\\
\omegae(m-l,x), &m> l,
\end{cases}
\end{equation}
where $0\leq k\leq l$ and $(\omegasl,\Psl)$ is another independent copy of $(\omega,\P)$. The superscirpts ``sl'' indicates ``slab''. In words, a slab with new environment $\omegasl$ is inserted in the time interval $(k,l]$ and $\omegae$ is shifted accordingly. 


The following lemma shows that both $\A_{n,k}$ and 
$\AA_{n,k}$ 
have high probability.

\begin{lemma}
\label{lem:AisTypical}
For any $\epsilon>0$, there exists $c_{\cref{lem:AisTypical}}>0$ such that the following hold for all $p> \pcr+\epsilon$ and $n\in\N$:
\begin{enumerate}
 \item[(i)] for all $\ell_n^2\leq k\leq n$ and $\omega \in \GGG_{\ell_n^2}$,
\begin{align}
\label{eq:boundA_1}
&\P\left(
\A_{n,k}\cond \mathcal F_{\ell_n^2}\right)(\omega) \ge 1-e^{-c_{\cref{lem:AisTypical}}\ell_n^2};
\end{align}
\item[(ii)] for all $k\le n$ and $\omegas$, 
\begin{align}
\label{eq:boundA_2}
&\Pe\left([\omegas,\omegae]_k\in 
\AA_{n,k}\right) \ge 1-e^{-c_{\cref{lem:AisTypical}}\ell_n^2}.
\end{align}
\end{enumerate}
\end{lemma}

\begin{proof}[Proof of Lemma \ref{lem:AisTypical}]
First, in the case $k\le2\ell_n^2$, we choose $j=k$ in \eqref{eq:def_A} and apply Lemma~\ref{lem:good}(ii). Next, we consider~\eqref{eq:boundA_1} in the case $k> 2\ell_n^2$. In this case, we choose $j=\ell_n^2$ and 
note that
\begin{align*}
\bigcap_{x\in[-k+\ell_n^2,k-\ell_n^2]^d}\Big\{\text{Either }\omega\in\G_{\ell_n^2,\ell_n^2+1}^{k-\ell_n^2,x}
\text{ or } (k-\ell_n^2,x)\nleftrightarrow(k,\Z^d)
\Big\}\subset\A_{n,k}
\end{align*}
since $\pi(k-\ell_n^2)\in [-k+\ell_n^2,k-\ell_n^2]^d$ for any path $\pi$ appearing in \eqref{eq:def_A}. Note that the left-hand side is independent of $\F_{\ell_n^2}$ and the probability of the events in the intersection does not depend on $k$ and $x$. Thus we can estimate the probability of the complement of each event by
\begin{align*}
&\P\left(\{(0,0)\leftrightarrow(\ell_n^2,\Z^d)\}\setminus\G_{\ell_n^2,\ell_n^2+1}\right)\\
&\quad\leq \P\left((0,0)\leftrightarrow(\ell_n^2,\Z^d),(0,0)\nleftrightarrow \infty\right)+\P\left(\{(0,0)\leftrightarrow\infty\}\setminus\G_{\ell_n^2,\ell_n^2+1}\right).
\end{align*}
The conclusion then follows from Theorem \ref{thm:largefinite}, Lemma \ref{lem:good}(i) and the union bound. The second assertion \eqref{eq:boundA_2} follows in the same way.
\end{proof}

\begin{proof}[Proof of Lemma \ref{lem:mart_diff}]
Recalling the idea explained in Section~\ref{sec:difficulty}, we write 
\begin{align*}
\Delta_k =\Esle\left[\log\widetilde N_n([\omegas,\omegae]_k)-\log\widetilde N_n([\omegas,\omegasl,\omegae]_{k-1,k+\ell_n^4})\right],
\end{align*}
where $\Esle=\Esl\otimes\Ee$. This is the coupling mentioned in Section \ref{sec:difficulty}. We show that on an event with high $\Psle$-probability, most of the $[\omegas,\omegae]_k$-open paths can be matched with $[\omegas,\omegasl,\omegae]_{k-1,k+\ell_n^4}$-open paths, from which \eqref{eq:Delta<} follows by taking expectation. The extra slab makes it possible to recover an open path by repairing the defect caused by changing the environment at time $k$. To make this precise, for fixed $\omegas$ and $\omegae$, we define the event
\begin{align}
\AAA_{n,k}(\omegas,\omegae)
:=\left\{\omegasl \colon
\begin{array}{l}
\text{For any open path $\pi_1$ from }(0,0)\text{ to }(n,\Z^d)\\
\text{in }[\omegas,\omegae]_k,\text{ there exists an open path $\pi_2$}\\
\text{from }(0,0)\text{ to }(n,\Z^d)\text{ in }[\omegas,\omegasl,\omegae]_{k-1,k+\ell_n^4}\\
\text{such that }\pi_1(j)=\pi_2(j)\text{ for all }j\le (k-1-2\ell_n^2)_+\\
\text{and }\pi_1(j)=\pi_2(j+\ell_n^4)\text{ for all }k+\ell_n^2\le j \le n-\ell_n^4.
\end{array}
\right\}.\label{eq:def_AAA}
\end{align}
On this event, all open paths in $[\omegas,\omegae]_k$ of length $n$ can be mapped to open paths in $[\omegas,\omegasl,\omegae]_{k-1,k+\ell_n^4}$. In the case $k\leq n-\ell_n^4-\ell_n^2$, thanks to the properties of $\pi_2$ in the third and fourth lines above, two open paths $\pi_1$ and $\pi_1'$ in $[\omegas,\omegae]_k$ cannot be mapped to the same path unless $\pi_1(j)=\pi_1'(j)$ for all $j\le (k-1-2\ell_n^2)_+$ and all $k+\ell_n^2 \le j \le n-\ell_n^4$. Therefore, on $\AAA_{n,k}(\omegas,\omegae)$, each open path in $[\omegas,\omegasl,\omegae]_{k-1,k+\ell_n^4}$ has at most $(2d+1)^{3\ell_n^2+\ell_n^4+1}$ pre-images and hence we have 
\begin{equation}
\label{eq:OnA_3}
\log \widetilde{N}_n([\omegas,\omegae]_k)- \log \widetilde{N}_n([\omegas,\omegasl,\omegae]_{k-1,k+\ell_n^4}) \le 2\ell_n^4 \log (2d+1)
\end{equation}
for all sufficiently large $n$. In the case $k > n-\ell_n^4-\ell_n^2$, the last line in~\eqref{eq:def_AAA} is void and then a similar argument applies. Next we need to show that $\AAA_{n,k}(\omegas,\omegae)$ has a high probability if $\omegas$ and $\omegae$ are ``well-connected'' environments in the sense of Definition~\ref{def:connected}.

\begin{lemma}
\label{lem:repair}
For any $\eps>0$, there exists $c_{\cref{lem:repair}}>0$ such that for any $p> \pcr+\eps$, $n\in\N$, $\ell_n^2+1\le k \le n$, $\omegas\in \A_{n,k-1}$ and $\omegae$ such that $[\omegas,\omegae]_k\in \AA_{n,k}$, the following holds:
\begin{equation}
 \Psl\left(\omegasl\notin \AAA_{n,k}(\omegas,\omegae)\right) \le e^{-c_{\cref{lem:repair}}\ell_n^2}.
\end{equation} 
\end{lemma}

We postpone the proof of Lemma~\ref{lem:repair} and complete the proof of Lemma \ref{lem:mart_diff} first. Let $\ell_n^2+1 \le k \le n$ and $\omegas \in \A_{n,k-1}$. By using~\eqref{eq:trivial} and~\eqref{eq:OnA_3}, we obtain
\begin{align*}
\Delta_k &= \Esle\left[\log\widetilde N_n([\omegas,\omegae]_k)-\log\widetilde N_n([\omegas,\omegasl,\omegae]_{k-1,k+\ell_n^4})\right]\\ 
&\le 2\ell_n^4 \log (2d+1)\Psle\left([\omegas,\omegae]_k\in \AA_{n,k}, \omegasl\in \AAA_{n,k}(\omegas,\omegae)\right)\\
&\qquad+n \log (2d+1)\Pe\left([\omegas,\omegae]_k\notin \AA_{n, k}\right)\\
&\qquad+n \log (2d+1)\Psle\left(\omegasl\notin \AAA_{n,k}(\omegas,\omegae),[\omegas,\omegae]_k\in \AA_{n,k}\right). 
\end{align*}
The probabilities in the final two lines are bounded by $e^{-c\ell_n^2}$ by Lemmas \ref{lem:AisTypical}(ii) and \ref{lem:repair}, so from our choice \eqref{eq:ell} of $\ell_n$, we can conclude that $\Delta_k\le 3\ell_n^4\log (2d+1)$ for all sufficiently large $n$. 
\end{proof}

\begin{proof}[Proof of Lemma \ref{lem:repair}]
Under the assumption $\omegas\in \A_{n,k-1}$, for any $[\omegas,\omegae]_k$-open path $\pi_1$ from $(0,0)$ to $(n,\Z^d)$, there exists $j\in[\ell_n^2,2\ell_n^2\wedge (k-1)]$ such that $(k-1-j,\pi_1(k-1-j))$ lies in the set
\begin{align}
\I(\omegas):= \left\{(k-1-m,x): m\in [\ell_n^2,2\ell_n^2\wedge (k-1)], x\in [-k,k]^d, \omegas \in G^{k-1-m, x}_m\right\}.
\end{align}
Note that by the definition of $G^{k-1-m, x}_m$ and the union bound, we have
\begin{equation}
 \begin{split}
&\Psl\left(\bigcap_{(k-1-m,x)\in\I(\omegas)}\left\{(k-1-m,x)\leftrightarrow (k+\ell_n^4,\Z^d)\text{ in }[\omegas,\omegasl,\omegae]_{k-1,k+\ell_n^4}\right\}\right)\\
&\quad \ge 1- e^{-c \ell_n^2}
 \end{split}
\label{eq:fw_repair}
\end{equation}
for all sufficiently large $n$. Since this event is $\F_{k+\ell_n^4}$-measurable, this bound is independent of $\omegae$. When $k>n-\ell_n^4$, this event ensures the existence of $\pi_2$ required in~\eqref{eq:def_AAA}. 

Thus we focus on the case $k\le n-\ell_n^4$. Under the assumption $[\omegas,\omegae]_k\in \AA_{n,k}$, for any $[\omegas,\omegae]_k$-open path $\pi_1$ from $(0,0)$ to $(n,\Z^d)$, the point $\pi_1(k+\ell_n^2)$ lies in the set
\begin{align}
\J(\omegae):= \left\{x\in [-k-\ell_n^2,k+\ell_n^2]^d: \omegae \in \GG^{\ell_n^2, x}_{\ell_n^2-1}\right\},
\end{align}
and the same argument as above yields
\begin{equation}
 \begin{split}
&\Psl\left(\bigcap_{y\in\J(\omegae)}\left\{(k-1,\Z^d) \leftrightarrow (k+\ell_n^4+\ell_n^2,y)\text{ in }[\omegas,\omegasl,\omegae]_{k-1,k+\ell_n^4}\right\} \right)\\
&\quad \ge 1- e^{-c\ell_n^2}
 \end{split}
\label{eq:bw_repair}
\end{equation}
for all sufficiently large $n$. 

The above two bounds ensure that for $\pi_1$ and $j$ appearing in the definition of $\A_{n,k-1}$, both $(m,\pi_1(m))_{m=0}^{k-1-j}$ and $(m,\pi_1(m+\ell_n^4))_{m=k+\ell_n^2}^n$ can be extended to an open path that crosses the slab with very high probability. We need one more bound to bridge these two paths. By Theorem~\ref{thm:coupled} and the union bound, we have
\begin{equation}
 \Psl\left(\bigcap_{z\in [-k,k]^d: (0,z)\leftrightarrow (\ell_n^4+1,\Z^d)}\C^{0,z}_{\ell_n^4+1}\right)\ge 1-e^{-c\ell_n^4}
\label{eq:slab_repair}
\end{equation}
for all sufficiently large $n$. Note that the condition $\omegasl\in \C^{0,z}_{\ell_n^4+1}$ implies $[\omegas,\omegasl,\omegae]_{k-1,k+\ell_n^4} \in \C^{k-1,z}_{\ell_n^4+1}$. 
Now if $\omegasl$ is taken from the intersection of the events in~\eqref{eq:fw_repair},~\eqref{eq:bw_repair} and~\eqref{eq:slab_repair}, then for any path $\pi_1$ connecting $(0,0)$ to $(n,\Z^d)$ in $[\omegas,\omegae]_k$, we can find $j\in [\ell_n^2,2\ell_n^2\wedge (k-1)]$, $x' \in [-k,k]^d$ and $ y'\in\Z^d$ such that
\begin{align}
&(k-1-j,\pi_1(k-1-j))\leftrightarrow (k-1,x')\leftrightarrow (k+\ell_n^4,\Z^d), \label{eq:slab1}\\
&(k-1,\Z^d)\leftrightarrow (k+\ell_n^4,y')\leftrightarrow (k+\ell_n^4+\ell_n^2,\pi_1(k+\ell_n^2)),\label{eq:slab2}
\end{align}
and then $\C^{k-1,x'}_{\ell_n^4+1}$ holds, all in the environment $[\omegas,\omegasl,\omegae]_{k-1,k+\ell_n^4}$. Since $|x'-y'|\leq 4\ell_n^2\ll \ell_n^4$, it follows that $(k+\ell_n^4,y')$ lies in the coupled zone for the percolation starting from $(k,x')$ for all sufficiently large $n$, see Figure~\ref{fig:repair}. Then by the definition of the coupled zone together with~\eqref{eq:slab1} and~\eqref{eq:slab2}, we have
\begin{align}
\label{eq:recovery_path}
(k-1-j,\pi_1(k-1-j))\leftrightarrow(k-1,x')\leftrightarrow(k+\ell_n^4,y')\leftrightarrow (k+\ell_n^4+\ell_n^2,\pi_1(k+\ell_n^2)) 
\end{align}
in $[\omegas,\omegasl,\omegae]_{k-1,k+\ell_n^4}$. This ensures the existence of $\pi_2$ required in~\eqref{eq:def_AAA}. 
\end{proof}
\begin{figure}[h]
\includegraphics[width=\textwidth]{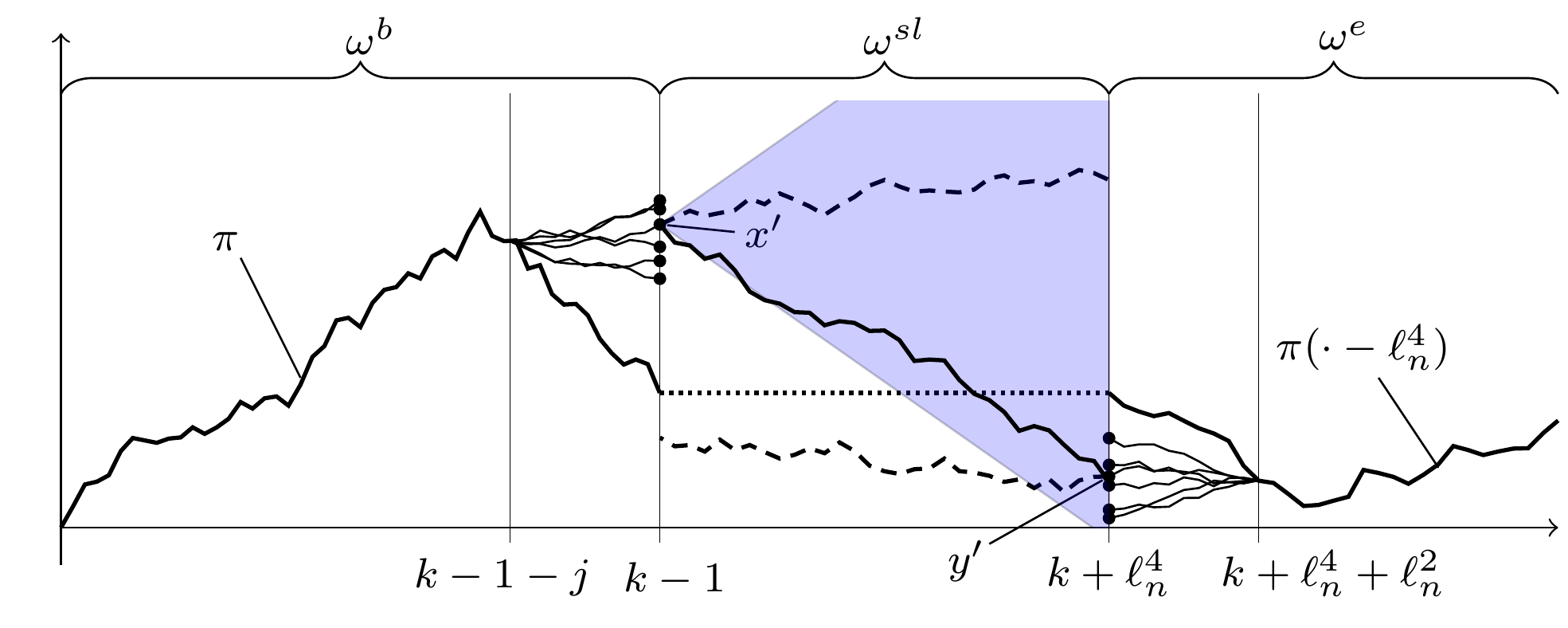} 
\caption{The repairing path for an $[\omegas,\omegae]_k$-open path $\pi$ in $[\omegas,\omegasl,\omegae]_{k-1,k+\ell_n^4}$. The  assumptions of Lemma \ref{lem:repair} intuitively means that there are many alternative paths from $(k-1-j,\pi(k-1-j))$ to level $k-1$ and from $(k+\ell_n^4+\ell_n^2,\pi(k+\ell_n^2))$ to level $k+\ell_n^4$, where $j\approx\ell_n^2$. With very high $\Psl$-probability, those alternative paths are further connected to levels $k+\ell_n^4$, resp. to level $k-1$ (dashed paths), and the percolation cone from level $k-1$ (shaded region) grows in linear speed. By making the slab large enough, we ensure that the percolation cone from $(k-1,x')$ covers $(k+\ell_n^4,y')$. The definition of the coupled zone then guarantees the existence of the repairing path (solid path in $\omegasl$).}
\label{fig:repair}
\end{figure}
\begin{lemma}
\label{lem:mart_diff2}
There exists $c_{\ref{lem:mart_diff2}}$ such that for every $n\in\N$, $\ell_n^2+1\leq k\leq n$ and $\omega \in \A_{n,k-1}$,
\begin{align}
\Delta_k\ge -c_{\ref{lem:mart_diff2}} \ell_n^4\log (2d+1).
\end{align}
\end{lemma}
\begin{proof}
The argument follows almost the same route as the proof of Lemma~\ref{lem:mart_diff}. Thus we only indicate the differences. We write the martingale difference as
\begin{align*}
-\Delta_k &= \Esle\left[\log\widetilde N_n([\omegas,\omegae]_{k-1})-\log\widetilde N_n([\omegas,\omegasl,\omegae]_{k,k+\ell_n^4+1})\right].
\end{align*}
We need an upper bound on this, conditionally on $\omegas\in\A_{n,k-1}$ and $\omegae\in\AA_{n,0}$. As before, we show that with high $\Psl$-probability every $[\omegas,\omegae]_{k-1}$-open path can be mapped to a $[\omegas,\omegasl,\omegae]_{k,k+\ell_n^4+1}$-open repairing path  satisfying the same constraints as in \eqref{eq:def_AAA}. The only difference is that the repairing path has to pass through the environment $\omegas|_{\{k\}\times\Z^d}$, which is not present in $[\omegas,\omegae]_{k-1}$. But the definition \eqref{eq:def_A} of $\A_{n,k-1}$ already ensures that $\omegas|_{\{k\}\times\Z^d}$ does not cut off too many open paths. More precisely, conditionally on $\omegas\in\A_{n,k-1}$ and $\omegae\in\AA_{n,0}$, for any $[\omegas,\omegae]_{k-1}$-open path $\pi_1$, we can find $j\in[\ell_n^2,2\ell_n^2]$ such that 
\begin{align*}
\omegas\in\G_{j+1}^{k-1-j,\pi_1(k-1-j)}\quad\text{ and }\quad\omegae\in\GG_{\ell_n^2}^{\ell_n^2,\pi_1(k+1+\ell_n^2)}.
\end{align*}
Now the same argument as in Lemma \ref{lem:repair} gives a lower bound on the $\Psl$-probability that repairing paths exist. 
\end{proof}
\begin{proof}
[Proof of Lemma~\ref{lem:conc_1}]
Applying Theorem~\ref{thm:Burkholder} to our martingale \eqref{eq:M_n}, we find that on $\omega\in\GGG_{\ell_n^2}$,
\begin{align}
\label{eq:Bapplied}
\E\left[\left|\log\widetilde N_n-\E[\log\widetilde N_n\mid \F_{\ell_n^2}]\right|^{2q}\cond \F_{\ell_n^2}\right]
\le c_{\cref{thm:Burkholder}}\E\left[\left(\sum_{k=\ell_n^2+1}^n\Delta_k^2\right)^q\cond\F_{\ell_n^2}\right].
\end{align} 
First, on the event $\bigcap_{k=\ell_n^2+1}^n \A_{n,k-1}$, we use Lemmas~\ref{lem:mart_diff} and~\ref{lem:mart_diff2} to get, on $\GGG_{\ell_n^2}$,
\begin{equation*}
\label{eq:on_well_conn}
\E\left[\left(\sum_{k=\ell_n^2+1}^n\Delta_k^2\right)^q \1_{\bigcap_{k={\ell_n^2+1}}^n \A_{n,k-1}} \cond\F_{\ell_n^2}\right] \le c n^q \ell_n^{8q}.
\end{equation*}
Second, on the complementary event $\bigcup_{k=\ell_n^2+1}^n(\A_{n,k-1})^c$, we use~\eqref{eq:trivial}, Lemma~\ref{lem:AisTypical}(i) and the union bound to get, on $\omega\in\GGG_{\ell_n^2}$,  
\begin{equation*}
\label{eq:outside_well_conn}
\begin{split}
\E\left[\left(\sum_{k=\ell_n^2+1}^n\Delta_k^2\right)^q{\1_{\bigcup_{k=\ell_n^2+1}^n (\A_{n,k})^c}}\cond \F_{\ell_n^2}\right]\le c n^{3q+1}e^{-c_{\cref{lem:AisTypical}}\ell_n^2}.
\end{split}
\end{equation*}
Since we choose $\ell_n \gg \log n$, this is negligible compared to the previous case. 

Substituting these bounds to~\eqref{eq:Bapplied} and using the Markov inequality, we can obtain the desired bound. 
\end{proof}

\subsubsection{Getting rid of the conditioning at zero temperature}
In this section, we prove Lemma~\ref{lem:conc_2}.
\begin{proof}[Proof of Lemma \ref{lem:conc_2}]
Using \eqref{eq:trivial}, Lemma \ref{lem:good} and Definition \ref{def:good}, we see that for all $n$ large enough,
\begin{align*}
\left|\E[\log\widetilde N_n\mid \GGG_{\ell_n^2}]-\E[\log\widetilde N_n\mid (0,0)\leftrightarrow(n,\Z^d)]\right|\leq 1.
\end{align*}
Let $(\omegas,\Ps)$ and $(\omegass,\Pss)$ denote independent copies of $(\omega,\P)$. It is now enough to show that for all $\omegas,\omegass\in \GGG_{\ell_n^2}$ and all $n$ large enough,
\begin{align}\label{eq:omega1omega3}
\left|\E[\log \widetilde N_n\mid\F_{\ell_n^2}](\omegas)-\E[\log \widetilde N_n\mid\F_{\ell_n^2}](\omegass)
\right|\leq 4\ell_n^4\log (2d+1).
\end{align}
This can be proved in a similar way to Lemma~\ref{lem:mart_diff}. The trick is to introduce additional independent copies $(\omegae,\Pe)$ and $(\omegasl,\Psl)$ and write
\begin{align*}
\E[\log \widetilde N_n\mid\F_{\ell_n^2}](\omegas)&=\Esle[\log \widetilde N_n([\omegas,\omegae]_{\ell_n^2})]\\
\E[\log \widetilde N_n\mid\F_{\ell_n^2}](\omegass)&=\Esle[\log \widetilde N_n([\omegass,\omegasl,\omegae]_{\ell_n^2,\ell_n^4+\ell_n^2})].
\end{align*}
We consider the event that every $[\omegas,\omegae]_{\ell_n^2}$-open path $\pi_1$ can be mapped onto a $[\omegass,\omegasl,\omegae]_{\ell_n^2,\ell_n^2+\ell_n^4}$-open path $\pi_2$ satisfying similar constraints as in \eqref{eq:def_AAA}. A straightforward modification of Lemma \ref{lem:repair} shows that this event has probability at least $1-e^{-c\ell_n^2}$ on $\omegas,\omegass\in\GGG_{\ell_n^2}$ and $\omegae\in\AA_{n,0}$. 
Moreover, since the multiplicity of the mapping is small, we have
\begin{equation}
\log\widetilde N_n([\omegas,\omegae]_{\ell_n^2})-\log\widetilde N_n([\omegass,\omegasl,\omegae]_{\ell_n^2,\ell_n^2+\ell_n^4})\le 2\ell_n^4\log (2d+1)
\end{equation}
for all $n$ large enough. By taking expectation over $(\omegasl,\omegae)$ and using Lemma \ref{lem:AisTypical}(ii) and \eqref{eq:trivial}, we obtain~\eqref{eq:omega1omega3} without the absolute value sign. Since the other bound follows by symmetry, we are done. 
\end{proof}

\subsection{Concentration inequality in positive temperature}
\label{sec:conc_positive}

In order to prove the continuity of $\f(\beta,p)$ at $\beta=\infty$, we need a concentration around the mean for $\log \widetilde{Z}^\beta_n$ uniformly in $\beta$, where $\widetilde{Z}^\beta_n:=Z^\beta_n+\1_{\{(0,0)\not\leftrightarrow(n,\Z^d)\}}$. The general idea of the proof is similar to the case $\beta=\infty$: we use the inserting slab trick to repair a path in order to bound the martingale differences. 

However, there is one point where we need a non-trivial modification. In the case $\beta=\infty$, we only needed to repair open paths and it made the proof of Lemma~\ref{lem:AisTypical} simple. For $\beta<\infty$, we need to repair non-open paths and this prevent us from using the existing results for oriented percolation, such as~Theorem~\ref{thm:largefinite}. More precisely, we need to find a repairing path as in \eqref{eq:recovery_path}, but in positive temperature, the path $\pi_1$ is not necessarily open. Then there is no guarantee that the start and end points $(k-1-j,\pi_1(k-1-j))$ and $(k+\ell_n^2,\pi_1(k+\ell_n^2))$ are percolating. At zero temperature, this was a crucial ingredient to ensure that many alternative paths are available.

We will deal with this issue by changing the definition of $\A_{n,k}$, $\AA_{n,k}$ and $\AAA_{n,k}$. The basic idea is that, if a path goes through many closed sites, then a repairing path is allowed to go through the same number of closed sites and thus easier to find. 
To make this precise, we introduce the notation $H_I(\omega,\pi):=\sum_{i\in I}\omega(i,\pi(i))$ for the energy (recall \eqref{eq:H_n}) restricted to $I\subseteq\Z_+$ and define the following events for $\ell_n^2\leq k\leq n $:
\begin{align}
\label{eq:A'_1}
\Ap_{n,k}&=\left\{\omega \colon 
\begin{array}{l}
\text{For any path }\pi\text{ from } (0,0)\text{ to } (n,\Z^d),\text{ either }\omega\in\G_{j,j+1}^{k-j,\pi(k-j)}
\\
\text{for some }j \in [\ell_n^2,2\ell_n^2\wedge k]\text{ or }H_{(k-2\ell_n^2,k-\ell_n^2]}(\pi)\ge \ell_n\text{ and }\\
\omega\in\G_{2\ell_n^2-\ell_n,2\ell_n^2-\ell_n+1}^{k-2\ell_n^2+\ell_n,x}
\text{ for some }x\in \pi(k-2\ell_n^2)+[-\ell_n,\ell_n]^d.
\end{array}
\right\},\\
\label{eq:A'_2}
\AAp_{n,k}&=\left\{\omega \colon
\begin{array}{l}
\text{For any path }\pi\text{ from } (k,[-n,n]^d)\text{ to } (k+2\ell_n^2,\Z^d),\text{ either }\\
\omega\in \GG_{j-1}^{k+j,\pi(k+j)}\text{ for some }j\in[\ell_n^2,2\ell_n^2]\text{ or }H_{(k+\ell_n^2,k+2\ell_n^2]}(\omega,\pi)\ge \ell_n\\
\text{and }\omega\in\GG_{2\ell_n^2-\ell_n-1}^{k+2\ell_n^2-\ell_n,x}\text{ for some }x\in \pi(k+2\ell_n^2)+[-\ell_n,\ell_n]^d.
\end{array}
\right\}, 
\end{align}
where in the case $k < 2\ell_n^2$, we drop the second condition in~\eqref{eq:A'_1}. Note that if $\omega\in G_j^{k-j,\pi(k-j)}$ for some $j\in[\ell_n^2,2\ell_n^2\wedge k]$, then we can expect an open repairing path starting from $(k-j,\pi(k-j))$ as in the zero temperature case. The event $\Ap_{n,k}$ ensures that if there is no such $j$, then $\pi$ visits many closed sites and we can expect that an open repairing path starts close by. In either case, the repairing path has lower energy and hence the cost of switching from $\pi$ is independent of $\beta$. Next, for fixed environments $\omegas$ and $\omegae$ and $\ell_n^2+1\leq k\leq n$, let
\begin{align}
\AAAp_{n,k}(\omegas,\omegae)
=\left\{\omegasl \colon
\begin{array}{l}
\text{For any path $\pi$ from }(0,0)\text{ to }(n,\Z^d), \text{ there }\\
\text{exists a path }\pi'\text{ from }(0,0)\text{ to }(n,\Z^d)\text{ such that }\\
H_n([\omegas,\omegasl,\omegae]_{k-1,k+\ell_n^4},\pi')\le H_n([\omegas,\omegae]_k,\pi),\\
\pi(j)=\pi'(j)\text{ for all } j\le (k-1-2\ell_n^2)_+
\text{ and }\\
\pi(j)=\pi'(j+\ell_n^4)\text{ for all } k+\ell_n^2\leq j\leq n-\ell_n^4.
\end{array}
\right\}.\label{eq:A'_3}
\end{align}
When $\omegasl\in \AAAp_{n,k}(\omegas,\omegae)$, any path $\pi$ in $[\omegas,\omegae]_k$ can be mapped to another path $\pi'$ that has a lower energy in $[\omegas,\omegasl,\omegae]_{k-1,k+\ell_n^4}$. In addition, the mapping $\pi\mapsto \pi'$ satisfies the same bound on the number of preimages as in~\eqref{eq:OnA_3}. As a result, for $\omegasl\in \AAAp_{n,k}(\omegas,\omegae)$, we have
\begin{equation}
\label{eq:OnA_3'}
\log \widetilde{Z}_n^\beta([\omegas,\omegae]_k)- \log \widetilde{Z}_n^\beta([\omegas,\omegasl,\omegae]_{k-1,k+\ell_n^4}) \le 2\ell_n^4 \log (2d+1)
\end{equation}
just as in the case $\beta=\infty$. The corresponding lower bound can be obtained by the same modification as in Lemma~\ref{lem:mart_diff2}. Therefore, in order to obtain the conditional concentration as in Lemma~\ref{lem:conc_1}, it remains to prove the following two lemmas which replace Lemmas~\ref{lem:AisTypical} and~\ref{lem:repair}. 
\begin{lemma}
\label{lem:A'isTypical}
For any $\epsilon>0$, there exists $c_{\cref{lem:A'isTypical}}>0$ such that for any $p> \pcr+\epsilon$, $n\in\N$, $\omega\in\GGG_{\ell_n^2}$ and $\ell_n^2\leq k\leq n$,
\begin{align}
\label{eq:boundA'_1}
&\P\left(\Ap_{n,k}\cond\F_{\ell_n^2}\right)(\omega) \ge 1-e^{-c_{\cref{lem:A'isTypical}}\ell_n},
\end{align}
and for any $k\le n$ and $\omegas$, 
\begin{align}
\label{eq:boundA'_2}
&\Pe\left([\omegas,\omegae]_k\in\AAp_{n,k}\right) \ge 1-e^{-c_{\cref{lem:A'isTypical}}\ell_n}.
\end{align}
\end{lemma}
The error terms in \eqref{eq:boundA'_1} and \eqref{eq:boundA_2} are worse than those in \eqref{eq:boundA'_1} and \eqref{eq:boundA'_2}, but since we take $\ell_n=\lfloor(\log n)^2\rfloor$, they decay faster than any power of $n$, which is enough for our arguments.
\begin{lemma}
\label{lem:repair'}
For any $\epsilon>0$, there exists $c_{\cref{lem:repair'}}>0$ such that for any $p> \pcr+\epsilon$, $n\in\N$, $\ell_n^2+1\le k \le n$, $\omegas\in \Ap_{n,k-1}$ and $\omegae$ such that $[\omegas,\omegae]_k\in \AAp_{n, k}$, 
\begin{equation}
 \Pe\left(\omegasl\notin \AAAp_{n,k}(\omegas,\omegae)\right) \le e^{-c_{\cref{lem:repair'}}\ell_n}
\end{equation} 
\end{lemma}

\begin{figure}
\includegraphics[width=.8\textwidth]{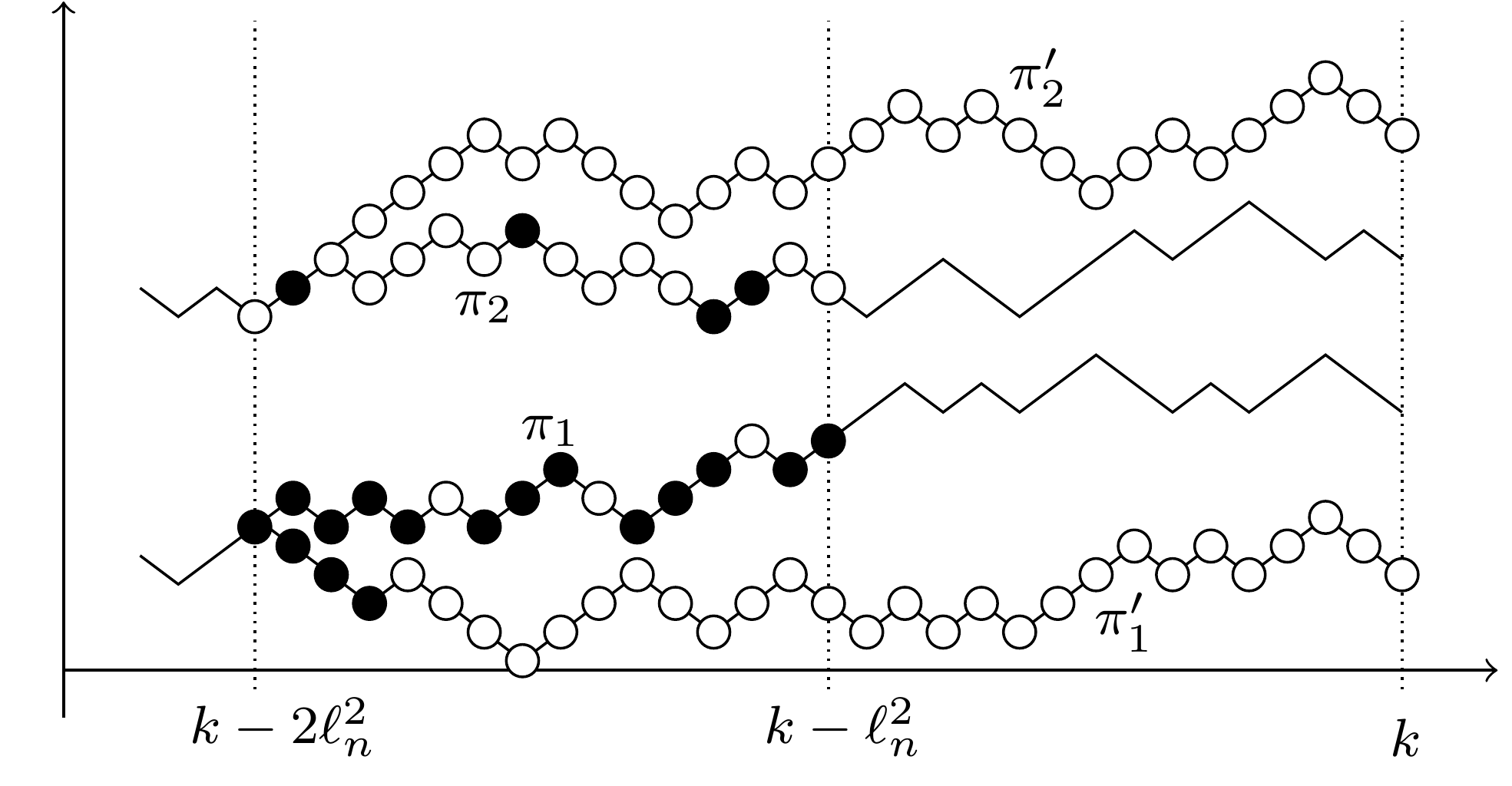}
\caption{The forward connection event $\Ap_{n,k}$ in positive temperature (empty/filled circles indicate open/closed sites). If a path visits many closed sites (like $\pi_1$ above), we can find an open path $\pi_1'$ close by and the cost of switching to $\pi_1'$ is independent of $\beta$. Otherwise the path must have a long subsequence where it only visits open sites (like $\pi_2$ above) and we can expect that an open path $\pi_2'$ branches off. Once $\pi_1'$ and $\pi_2'$ are found the abstract goodness condition has high probability.} \label{fig:positive} 
\end{figure} 

\begin{proof}
[Proof of Lemma~\ref{lem:A'isTypical}]
We refer to Figure \ref{fig:positive} for an illustration of the main idea. 
For $k\le 3\ell_n^2$, we take $j=k$ in \eqref{eq:A'_1} and apply Lemma \ref{lem:good}(ii) just as in the proof of Lemma~\ref{lem:AisTypical}. In particular, since we use only the first condition in~\eqref{eq:A'_1}, whether $k\ge 2\ell_n^2$ or not is irrelevant. 

For the case $k>3\ell_n^2$, note that by Theorem \ref{thm:largefinite} and Lemma \ref{lem:good}(i), for every $j\geq \ell_n$,
\begin{align*}
\P\left(\{(0,0)\leftrightarrow(\ell_n,\Z^d)\}\setminus \G_{j,j+1}\right)\le &\P\left((0,0)\leftrightarrow(\ell_n,\Z^d),(0,0)\nleftrightarrow\infty\right)\\
&+\P\left(\{(0,0)\leftrightarrow\infty\}\setminus \G_{j,j+1}\right)\\
\leq &e^{-c\ell_n}.
\end{align*}
Thus the union bound shows that with probability at least $1-e^{-c\ell_n}$, 
\begin{equation}
\label{eq:either}
\begin{split}
& \text{every }(t,x)\in[k-2\ell_n^2,k-\ell_n^2-\ell_n]\times[-n,n]^d \text{ satisfies}\\ 
& \text{either }\omega\in \G_{k-t,k-t+1}^{t,x}
\text{ or } (t,x)\nleftrightarrow (t+\ell_n,\Z^d).
\end{split}
\end{equation}
Note that \eqref{eq:either} is independent of $\F_{\ell_n^2}$. In addition, by Lemma \ref{lem:good}(i) and Theorem \ref{thm:large_initial},
\begin{align*}
&\P\left(\omega\notin \G_{2\ell_n^2-\ell_n,2\ell_n^2-\ell_n+1}^{0,x}
\text{ for every }x\in[-\ell_n,\ell_n]^d\right)\\
&\quad \le \P\left([-\ell_n,\ell_n]^d\nleftrightarrow \infty\right)+3^d\ell_n^d\P\left(\{(0,0)\leftrightarrow\infty\}\setminus \G_{2\ell_n^2-\ell_n,2\ell_n^2-\ell_n+1}
\right)\\
& \quad \le e^{-c\ell_n}.
\end{align*}
Using another union bound, with probability at least $1-e^{-c\ell_n}$, 
\begin{align}\label{eq:or}
\text{for every }x\in[-n,n]^d, 
\text{ there exists }x'\in x+[-\ell_n,\ell_n]^d\text{ with }\omega\in \G_{2\ell_n^2-\ell_n,2\ell_n^2-\ell_n+1}^{k-2\ell_n^2+\ell_n,x'}
.
\end{align}
Note that \eqref{eq:either}, resp. \eqref{eq:or}, guarantees the existence of repairing paths like $\pi_2'$, resp. $\pi_1'$, in Figure \ref{fig:positive}. We assume both~\eqref{eq:either} and~\eqref{eq:or} hold and let $\pi$ denote any path. If
\begin{align*}
H_{(k-2\ell_n^2,k-\ell_n^2]}(\omega,\pi)\leq \ell_n,
\end{align*}
then there exists $j\in[\ell_n^2+\ell_n,2\ell_n^2]$ such that $H_{(k-j,k-j+\ell_n]}(\omega,\pi)=0$. By \eqref{eq:either}, we therefore have $\omega\in G_{j,j+1}^{k-j,\pi(k-j)}$
as desired. If, on the other hand,
\begin{align*}
H_{(k-2\ell_n^2,k-\ell_n^2]}(\omega,\pi)>\ell_n,
\end{align*}
then \eqref{eq:or} implies that $\G_{2\ell_n^2-\ell_n,2\ell_n^2-\ell_n+1}^{k-2\ell_n^2+\ell_n,x'}$ 
holds for some $x'\in\pi(k-2\ell_n^2)+[-\ell_n,\ell_n]^d$. The proof of~\eqref{eq:boundA'_2} is similar and we omit it. 
\end{proof}

\begin{proof}
[Proof of Lemma~\ref{lem:repair'}]
Fix $\omegas,\omegae$ with $\omegas\in\Ap_{n,k-1}$ and $[\omegas,\omegae]_k\in  \AAp_{n, k}$ and let $\pi$ be a path starting from $(0,0)$. 

In the case $k>2\ell_n^2+1$ and $H_{(k-1-2\ell_n^2,k-1-\ell_n^2]}(\omega,\pi)\geq  \ell_n$, let $j_1:=2\ell_n^2-\ell_n$ and let $\pi_1\colon[0,k-1-j_1]\to\Z^d$ denote any (deterministically chosen) path with $\pi_1|_{[0,k-1-2\ell_n^2]}=\pi|_{[0,k-1-2\ell_n^2]}$ and $\pi_1(k-1-j_1)=x$, where $x$ is the site whose existence is guaranteed in the final line of \eqref{eq:A'_1}. In the case $k\le 2\ell_n^2+1$ or $H_{(k-1-2\ell_n^2,k-1-\ell_n^2]}(\omega,\pi_1)<\ell_n$, let $j_1$ denote the value of $j$ from the second line of \eqref{eq:A'_1} and let $\pi_1:=\pi|_{[0,k-1-j_1]}$. In both cases, we have ensured that $[\omegas,\omegae]_k\in G_{j_1}^{k-1-j_1,\pi_1(k-1-j_1)}$ and that
\begin{align*}
H_{(0,k]}([\omegas,\omegae]_k,\pi)\geq H_{(0,k-1-j_1]}([\omegas,\omegae]_k,\pi_1).
\end{align*}
Similarly, using the definition \eqref{eq:A'_2}, we find $j_2\in[\ell_n^2,2\ell_n^2\wedge k]$ and a path $\pi_2:[k+j_2,n]\to\Z^d$ such that $\pi_2|_{[k+2\ell_n^2,n]}=\pi|_{[k+2\ell_n^2,n]}$, $[\omegas,\omegae]_k\in \GG_{j_2-1}^{k+j_2,\pi_2(k+j_2)}$ and 
\begin{align*}
H_{(k,n]}([\omegas,\omegae]_k,\pi)\geq H_{(k+j_2,n]}([\omegas,\omegae]_k,\pi_2).
\end{align*}
Using $(k-1-j_1,\pi_1(k-1-j_1))$ and $(k+j_2+\ell_n^4,\pi_2(k+j_2+\ell_n^4))$ in place of $(k-1-j,\pi_1(k-1-j))$ and $(k+\ell_n^4+\ell_n^2,\pi_1(k+\ell_n^2))$ in \eqref{eq:slab1} and \eqref{eq:slab2}, we can now repeat the argument in the proof of Lemma \ref{lem:repair}. Namely, with $\Psl$-probability at least $1-e^{-c\ell_n^2}$, we can construct another path $\pi'$ with $\pi'|_{[0,k-1-j_1]}=\pi_1$, $\pi'_{[k+j_2+\ell_n^4,n]}=\pi_2|_{[k+j_2,n-\ell_n^4]}$ 
and such that the middle path $\pi'|_{(k-1-j_1,k+j_2+\ell_n^4]}$ is open in $[\omegas,\omegasl,\omegae]_{k-1,k+\ell_n^4}$. It is simple to check that this path satisfies all the properties in~\eqref{eq:A'_3}.  
\end{proof}
Once the idea for constructing the repairing paths is understood, the positive temperature counterpart of Lemma~\ref{lem:conc_2} can be proved in the same way as before. We leave the detail to the reader.
\section{Non-random fluctuations}\label{sec:nonrand}

\subsection{At zero temperature}\label{sec:nonrand_zero}

In this model, we do not have the super-additivity of the sequence $(a_n^\infty)_{n\in\N}$ defined by
\begin{align*}
a_n^\infty:=\E[\log N_n\mid (0,0)\leftrightarrow (n,\Z^d)]
\end{align*}
because of the conditioning. 
However, we do have an almost super-additivity and it suffices for our purpose. We use the idea in~\cite{Z10} for the first passage percolation, while that paper aims at an opposite bound. See the comment before Lemma~\ref{lem:almost_subadd}. For $n\leq m$ and $x,y\in\Z^d$, we let $N_{n,x;m,y}$ denote the number of open paths from $(n,x)$ to $(m,y)$.
\begin{lemma}
\label{lem:almost_superadd}
For every $\epsilon, \delta>0$, there exists $c_{\cref{lem:almost_superadd}}> 0$ such that for all $p> \pcr+\epsilon$ and $m,n\in\N$,
\begin{equation}
a_{n+m}^\infty\geq a_m^\infty+a_n^\infty-c_{\cref{lem:almost_superadd}}(m+n)^{1/2+\delta}.
\end{equation}
\end{lemma}
\begin{proof}
By symmetry we may assume $m\leq n$. Let $x^*=x^*(\omega;m)$ be the point in $(m,\Z^d)$ where we have the largest number of open paths from $(0,0)$, that is, 
\begin{align}\label{eq:xstar}
x^*(\omega;m)=\mathop{\text{argmax}}_{x\in\Z^d} N_{0,0;m,x}.
\end{align}
Since at most $(2m+1)^d$ points in $(m,\Z^d)$ can be connected to $(0,0)$, we have 
\begin{equation}
 N_{0,0;m,x^*}\ge \frac{N_m}{(2m+1)^d},
\end{equation}
and there exists $x\in \Z^d$ such that 
\begin{align*}
\P\left(x^*(\omega;m)=x \cond (0,0)\leftrightarrow (m,\Z^d)\right) \ge (2m+1)^{-d}.
\end{align*}
Let us consider the event 
\begin{equation}
\label{eq:E_mn}
E_{x,m,n}= \{(0,0)\leftrightarrow (m,\Z^d),x^*(\omega;m)=x\}\cap\{(m,x) \leftrightarrow (m+n,\Z^d)\}, 
\end{equation}
on which we have 
\begin{equation}
\begin{split}
\label{eq:superadd1}
\log N_{m+n} & \ge \log N_{0,0;m,x}+\log N_{m,x;m+n,\Z^d}\\
&\ge \log N_m+\log N_{m,x;m+n,\Z^d}-d\log(2m+1).
\end{split}
\end{equation}
Using independence, we have
\begin{equation}\label{eq:independence}
\begin{split}
\P\left(E_{x,m,n}\right) &= \P\left(x^*(\omega;m)=x, (0,0)\leftrightarrow (m,\Z^d)\right)\P\left((m,x)\leftrightarrow(m+n,\Z^d)\right)\\
&\ge(2m+1)^{-d} \P\left((0,0)\leftrightarrow \infty\right)^2.
\end{split}
\end{equation}
On the other hand, for $k\leq l$ and $y\in\Z^d$, let
\begin{align}
&E_{(k,y)\leftrightarrow(l,\Z^d)}^\infty:= \left\{(k,y)\nleftrightarrow (l,\Z^d)\text{ or }\left|\log N_{k,y;l,\Z^d}-a_{l-k}^\infty\right|\le (l-k)^{\frac12+\delta}\right\},\label{eq:m+n}
\end{align}
We know from Proposition~\ref{prop:conc_zero} that there exists $c>0$ such that 
\begin{align}\label{eq:uniform}
\P\left(E^\infty_{(0,0)\leftrightarrow (m,\Z^d)}\cap E^\infty_{(m,x)\leftrightarrow(m+n,\Z^d)}\cap E^\infty_{(0,0)\leftrightarrow(n+m,\Z^d)}\right)\geq 1-cm^{-d-1},
\end{align}
where we used $m\leq n$. In particular, by comparing with \eqref{eq:independence}, we see that for all sufficiently large $m$, the event
\begin{align*}
E_{x,m,n}\cap E^\infty_{(0,0)\leftrightarrow (m,\Z^d)}\cap E^\infty_{(m,x)\leftrightarrow(m+n,\Z^d)}\cap E^\infty_{(0,0)\leftrightarrow(n+m,\Z^d)}
\end{align*}
is non-empty. Since $E_{x,m,n}$ implies $(0,0)\leftrightarrow (m,\Z^d)$, $(m,x)\leftrightarrow  (m+n,\Z^d)$ and $(0,0)\leftrightarrow (m+n,\Z^d)$, we know that the second condition in \eqref{eq:m+n} must hold in the last three events above. For an environment $\omega$ in this intersection, we can replace the $\log$-terms in \eqref{eq:superadd1} with $a_n^\infty$, $a_m^\infty$ and $a_{m+n}^\infty$ and their respective error-terms, which gives the desired bound. 
\end{proof}

\begin{remark}
\label{rem:existence}
This proposition implies that for any $p>\pcr$,
\begin{align*}
\lim_{n\to\infty}\E[\log \widetilde{N}_n\mid (0,0)\leftrightarrow (n,\Z^d)]=\tilde\alpha_p
\end{align*}
exists, and Proposition~\ref{prop:conc_zero} together with the Borel--Cantelli lemma show that $\frac1n \log \widetilde{N}_n$ converges to $\tilde\alpha_p$ almost surely on $\{(0,0)\leftrightarrow \infty\}$. 
\end{remark} 

We turn to proving almost sub-additivity for the same sequence. We follow the argument in~\cite{Z10} again, which proves a similar result for the non-directed first passage percolation. It becomes much simpler for the directed models, as is done in~\cite[Section~3]{N18} for directed polymers with unbounded jumps. Here it gets even simpler due to the nearest neighbor nature of the model. We present a proof for the reader's convenience.

\begin{lemma}\label{lem:almost_subadd}
For every $\epsilon,\delta>0$, there exists $c_{\cref{lem:almost_subadd}}>0$ such that for all $p> \pcr+\epsilon$ and $n\in\N$,
\begin{align}
\label{eq:NRF_lower}
2a_n^\infty\ge a_{2n}^\infty-c_{\cref{lem:almost_subadd}}n^{\frac12+\delta}.
\end{align}
\end{lemma}

\begin{proof}
Let  $x^{**}=x^{**}(\omega;n,2n)$ be the point in $(n,\Z^d)$ where the largest number of paths connecting $(0,0)$ to $(2n,\Z^d)$ go through, i.e.,
\begin{align*}
x^{**}(\omega;n,2n):=\operatorname{arg\,max}_x N_{0,0;n,x}N_{n,x;2n,\Z^d}.
\end{align*}
Since at most $(2n+1)^d$ points in $(n,\Z^d)$ can be connected from $(0,0)$, we have 
\begin{equation}
\label{eq:x^**}
\begin{split}
N_{2n} & \le (2n+1)^d N_{0,0;n,x^{**}} N_{n,x^{**};2n,\Z^d} \le (2n+1)^dN_n N_{n,x^{**};2n,\Z^d}.
\end{split}
\end{equation}
Moreover, there exists $x\in\Z^d$ such that 
\begin{align}\label{eq:lower}
\P\left(x^{**}(\omega;n,2n)=x\cond (0,0)\leftrightarrow(2n,\Z^d)\right) \ge (2n+1)^{-d}.
\end{align}
Using the events defined in \eqref{eq:m+n}, we have, by Proposition~\ref{prop:conc_zero},
\begin{align*}
\P\left(E^\infty_{(0,0)\leftrightarrow (n,\Z^d)}\cap E^\infty_{(n,x)\leftrightarrow(2n,\Z^d)}\cap E^\infty_{(0,0)\leftrightarrow(2n,\Z^d)}\right)\geq 1-cn^{-d-1}
\end{align*}
In particular, comparing with \eqref{eq:lower}, we see that for all sufficiently large $n$, there exists 
\begin{align*}
\omega\in \{(0,0)\leftrightarrow(2n,\Z^d),x^{**}(\omega;n,2n)=x\}\cap E^\infty_{(0,0)\leftrightarrow (n,\Z^d)}\cap E^\infty_{(n,x)\leftrightarrow(2n,\Z^d)}\cap E^\infty_{(0,0)\leftrightarrow(2n,\Z^d)}.
\end{align*}
Using this $\omega$ in~\eqref{eq:x^**}, we can replace $x^{**}$ by $x$ and further substitute $a_n^\infty$ and $a_{2n}^\infty$ together with their error terms. Then the conclusion follows in the same way as in Lemma \ref{lem:almost_superadd}.
\end{proof}
\begin{proof}
[Proof of Theorem~\ref{thm:NRF} for $\beta=\infty$]
Lemmas~\ref{lem:almost_superadd} and~\ref{lem:almost_subadd} yields
\begin{equation}
\left|\frac{1}{n}\E[\log N_n\mid (0,0)\leftrightarrow (n,\Z^d)]- \frac{1}{2n}\E[\log N_{2n}\mid (0,0)\leftrightarrow (2n,\Z^d)]\right| \le cn^{-\frac12+\epsilon}.
\end{equation}
By making repeated use of this bound, we have
\begin{equation}
\begin{split}
& \left|\frac1n \E[\log N_n\mid (0,0)\leftrightarrow (n,\Z^d)]-\frac{1}{2^k n}\E[\log  N_{2^k n}\mid (0,0)\leftrightarrow (2^k n,\Z^d)]\right|\\
&\quad \le cn^{-\frac12+\epsilon}\sum_{j= 0}^{k-1} 2^{-(\frac{1}{2}+\epsilon)j}. 
\end{split}
\end{equation}
Since the above sum converges as $k\to\infty$ and since Theorem~\ref{thm:ggm} implies
\begin{align*}
  \frac{1}{2^k n}\E[\log  N_{2^k n}\mid (0,0)\leftrightarrow (2^k n,\Z^d)] \xrightarrow{k\to\infty}\tilde{\alpha}_p,
\end{align*}
we obtain the desired bound.
\end{proof}

\subsection{In positive temperature}

In this section we want to extend Lemmas~\ref{lem:almost_superadd} and~\ref{lem:almost_subadd} to positive temperature, i.e., to the sequence $(a_n^\beta)_{n\in\N}$ with
\begin{align*}
a_n^\beta:=\E[\log Z_n^\beta\mid (0,0)\leftrightarrow(n,\Z^d)].
\end{align*}
We write $Z^\beta_{m,x;n,y}$ for the positive temperature counterpart to $N_{m,x;n,y}$ introduced in the beginning of Section \ref{sec:nonrand_zero}. We again start by proving almost superadditivity. We show the following slightly weaker statement since  we do not intend to prove the existence of $\lim_{n\to\infty}\frac{1}{n}a_n^\beta$ for $\beta<\infty$.

\begin{lemma}\label{lem:almost_superadd_pos}
For every $\epsilon,\delta>0$, there exists $c_{\cref{lem:almost_superadd_pos}}> 0$ such that for all $p> \pcr+\epsilon$, $n\in\N$ and all $\beta\in[0,\infty]$,
\begin{equation}
\begin{split}
&a_{2n}^\beta \geq 2a_n^\beta-c_{\cref{lem:almost_superadd_pos}}n^{\frac12+\delta}.
\end{split}
\end{equation}
\end{lemma}

\begin{proof}
Replacing $\log N_{m,x;n,y}$ by $(\log Z_{m,x;n,y}^\beta)1_{(m,x) \leftrightarrow (n,\Z^d)}$, we can repeat the proof of Lemma~\ref{lem:almost_superadd} except that when $\beta<\infty$, the event $E_{x,n,n}$ does not guarantee $(0,0)\leftrightarrow (2n,\Z^d)$. However, by using Theorem~\ref{thm:largefinite}, we can show that
\begin{align*}
  \P(E_{x,n,n}, (0,0)\not\leftrightarrow (2n,\Z^d))
  \le e^{-cn}.
\end{align*}
Therefore, instead of \eqref{eq:independence}, we have
\begin{align*}
  \P(E_{x,n,n}, (0,0)\leftrightarrow (2n,\Z^d))
  \ge \frac{1}{2}(2n+1)^{-d}\P((0,0)\leftrightarrow\infty)^2    
\end{align*}
for all sufficiently large $n$, uniformly in $p> \pcr+\epsilon$. Since this and the positive temperature version of \eqref{eq:uniform} hold uniformly in $\beta$, so does the resulting bound. 
\end{proof}

The extension of Lemma \ref{lem:almost_subadd} to positive temperature is a bit more subtle. The problem is that we do not have~\eqref{eq:lower} since $x^{**}(\omega;n,2n)=x$ does not imply $(n,x)\leftrightarrow(2n,\Z^d)$ in positive temperature.
 
We deal with this issue by repairing the potential defect that occurs if $\{(n,x)\nleftrightarrow(2n,\Z^d)\}$, using a local surgery of the environment as in the proof of Theorem \ref{thm:conc}. The argument here is simpler because we do not have to deal with conditional expectations as in Lemma \ref{lem:mart_diff}. We give a direct proof that does not use any lemmas from Section \ref{sec:conc}.

\begin{lemma}\label{lem:almost_subadd_pos}
For any $\epsilon, \delta>0$, there exists $c_{\cref{lem:almost_subadd_pos}}>0$ such that for all $p> \pcr+\epsilon$, $n\in\N$ and all $\beta\in[0,\infty]$,
\begin{align*}
2a_{n}^\beta \ge a_{2n}^\beta-c_{\cref{lem:almost_subadd_pos}}n^{\frac12+\delta}.
\end{align*}
\end{lemma}

\begin{proof}
Recall the environments $\omegas,\omegasl,\omegae$ introduced before the proof of Lemma \ref{lem:mart_diff}. In this proof, we write $\P=\Ps\otimes\Psl\otimes\Pe$ for simplicity.
Our strategy is to show that on an event with not-too-small probability, there exist $x$ and $y\in x+[-\ell_n^2,\ell_n^2]^d$ such that
\begin{equation}\label{eq:strategy}
\begin{split}
&\frac{1}{(2n+1)^d}Z_{0,0;2n,\Z^d}^\beta([\omegas,\omegae]_n)\\
&\quad\leq Z_{0,0;n,x}^\beta([\omegas,\omegae]_n)Z_{n,x;2n;\Z^d}^\beta([\omegas,\omegae]_n)\\
&\quad\leq Z_{0,0;n+\ell_n^4,y}^\beta([\omegas,\omegasl,\omegae]_{n,n+2\ell_n^4}) \\&\quad\qquad \cdot Z_{n+\ell_n^4,y;2n+2\ell_n^4,\Z^d}^\beta([\omegas,\omegasl,\omegae]_{n,n+2\ell_n^4})(2d+1)^{2\ell_n^2+2\ell_n^4},
\end{split}
\end{equation}
and such that, with very high probability, the terms in the first and last line are close to $a_{2n}^\beta$ and $2a_{n+\ell_n^4}^\beta$ after taking logarithms. Let
\begin{align*}
x^{**}:=\operatorname{arg\,max}_x Z_{0,0;n,x}^\beta([\omegas,\omegae]_n)Z_{n,x;2n,\Z^d}^\beta([\omegas,\omegae]_n)
\end{align*}
and let $x\in\Z^d$ be such that
\begin{align}\label{eq:not_too_small_prob}
\P\left(x^{**}=x\cond(0,0)\leftrightarrow(2n,\Z^d)\text{ in }[\omegas,\omegae]_n\right)\geq \frac 1{(2n+1)^d}.
\end{align}
Then the first inequality in \eqref{eq:strategy} holds on $\{x^{**}=x\}$. Thus it remains to find an event with not-too-small probability on which the second inequality in \eqref{eq:strategy} is justified. We start by modifying the events $E_{(k,y)\leftrightarrow(l,\Z^d)}$ from Section \ref{sec:nonrand_zero} as follows: 
\begin{align}\label{eq:2n}
\tEa&:=\left\{ 
\begin{array}{l}
(0,0)\nleftrightarrow(2n,\Z^d)\text{ in }[\omegas,\omegae]_{n}\text{ or}\\
|\log Z^\beta_{2n}([\omegas,\omegae]_{n})-a_{2n}^\beta|\leq (2n)^{\frac 12+\delta}
\end{array}\right\},\\
\tEb&:=\left\{ 
\begin{array}{l}
(0,0)\nleftrightarrow(n+\ell_n^4,\Z^d)\text{ in }[\omegas,\omegasl,\omegae]_{n,n+2\ell_n^4}\text{ or}\\
|\log Z^\beta_{n+\ell_n^4}([\omegas,\omegasl,\omegae]_{n,n+2\ell_n^4})-a_{n+\ell_n^4}^\beta|\leq 2n^{\frac 12+\delta}
\end{array}\right\},\\ 
\tEc&:=\bigcap_{y\in[-2n,2n]^d}\left\{ 
\begin{array}{l} 
(n+\ell_n^4,y)\nleftrightarrow(2n+2\ell_n^4,\Z^d)\text{ in }[\omegas,\omegasl,\omegae]_{n,n+2\ell_n^4}\text{ or}\\
|\log Z^\beta_{n+\ell_n^4,y;2n+2\ell_n^4,\Z^d}([\omegas,\omegasl,\omegae]_{n,n+2\ell_n^4})-a_{n+\ell_n^4}^\beta|\leq 2n^{\frac 12+\delta}
\end{array}\right\}.\label{eq:n2n}
\end{align}
By Theorem \ref{thm:conc} we have
\begin{align}\label{eq:high_prob}
\P\left(\tEa\cap \tEb \cap \tEc\right)\geq 1-cn^{-d-1}
\end{align}
for some constant $c$ independent of $\beta$. Next, let
\begin{align}\label{eq:def_connect}
E^{\text{conn}}:=\left\{
\begin{array}{l}
\text{There exists }y\in x+[-\ell_n^2,\ell_n^2]\times\{0\}^{d-1}\text{ such that }(0,\Z^d)\leftrightarrow(n+\ell_n^4,y)\\
\text{and }(n+\ell_n^4,y)\leftrightarrow (2n+2\ell_n^4,\Z^d)\text{ in }[\omegas,\omegasl,\omegae]_{n,n+2\ell_n^4}.
\end{array}\right\}. 
\end{align}
By Theorem \ref{thm:large_initial}, it is easy to see that this event has very high probability:
\begin{align}\label{eq:high_prob1}
\P\left(E^{\text{conn}}\right)\geq 1-e^{-c\ell_n}.
\end{align}
Indeed, by dividing $x+[-\ell_n^2,\ell_n^2]\times\{0\}^{d-1}$ into sub-intervals of size $\ell_n$ and applying Theorem~\ref{thm:large_initial} to each sub-interval, we can find $\ell_n$ many forward percolation points in $x+[-\ell_n^2,\ell_n^2]\times\{0\}^{d-1}$, with probability more than $1-e^{-c\ell_n}$. Then, by applying Theorem~\ref{thm:large_initial} to this set of forward percolation points, we can find a backward percolation point with probability more than $(1-e^{-c\ell_n})^2$.

\smallskip We now define a mapping between paths going through $(0,0), (n,x)$ and $(2n,\Z^d)$ in $[\omegas,\omegae]_n$ and repairing paths going through $(0,0), (n+\ell_n^4,y)$ and $(2n+2\ell_n^4,\Z^d)$ in $[\omegas,\omegasl,\omegae]_{n,n+2\ell_n^4}$, where $y\in x+[-\ell_n^2,\ell_n^2]$ is as in \eqref{eq:def_connect}, using the same construction as in Lemmas \ref{lem:A'isTypical} and \ref{lem:repair'}. This ensures that not too many paths are mapped onto the same repairing path and that the mapping reduces the number of closed sites along the path. The first condition is important to control the error-term in \eqref{eq:strategy} while the second condition ensures that the error-term does not depend on $\beta$. Let 
\begin{align*}
\accentset{\to}E^{\text{repair}}&:=\left\{
\begin{array}{l}
\text{For every path }\pi\text{ from } (0,0)\text{ to }(2n,\Z^d),\text{ either there exists }\\
j\in[n-2\ell_n^2,n-\ell_n^2]\text{ such that }[\omegas,\omegasl,\omegae]_{n,n+2\ell_n^4}\in \C_{n+\ell_n^4-j}^{j,\pi(j)},\\
\text{or }H_{(n-2\ell_n^2,n-\ell_n^2]}([\omegas,\omegae]_{n},\pi)\geq \ell_n\text{ and there exists }\\
z\in \pi(n-2\ell_n^2)+[-\ell_n,\ell_n]^d \text{ such that }
[\omegas,\omegasl,\omegae]_{n,n+2\ell_n^4}\in \C_{\ell_n^4+2\ell_n^2-\ell_n}^{n-2\ell_n^2+\ell_n, z}.
\end{array}
\right\},\\
\accentset\leftarrow E^{\text{repair}}&:=\left\{
\begin{array}{l}
\text{For every path }\pi\text{ from } (0,0)\text{ to }(2n,\Z^d),\text{ either there exists }\\
j\in[n+\ell_n^2,n+2\ell_n^2]\text{ such that }[\omegas,\omegasl,\omegae]_{n,n+2\ell_n^4}\in
\accentset\leftarrow \C_{j-n+\ell_n^4}^{j+2\ell_n^4,\pi(j)},\\
\text{or }H_{(n+\ell_n^2,n+2\ell_n^2]}([\omegas,\omegae]_{n},\pi)\geq \ell_n\text{ and there exists}\\
z\in \pi(n+2\ell_n^2)+[-\ell_n,\ell_n]^d\text{ such that }[\omegas,\omegasl,\omegae]_{n,n+2\ell_n^4}\in \accentset\leftarrow\C_{\ell_n^4+2\ell_n^2-\ell_n}^{n+2\ell_n^4+2\ell_n^2-\ell_n, z}.
\end{array}
\right\},\\
E^{\text{repair}}&:= \accentset\rightarrow E^{\text{repair}}\cap \accentset\leftarrow E^{\text{repair}}.
\end{align*}
These definitions are similar to \eqref{eq:A'_1} and \eqref{eq:A'_2}, except that we require that the coupled zone from certain points grow regularly as in Definition~\ref{def:coupled}, instead of the abstract goodness condition from Definition \ref{def:good}. We can show that
\begin{align}\label{eq:high_prob2}
\P\left(E^{\text{repair}}\right)\geq 1-e^{-c\ell_n}
\end{align}
in a similar way to Lemma~\ref{lem:A'isTypical}, using Theorem \ref{thm:coupled} in place of Lemma~\ref{lem:good} for \eqref{eq:either} and \eqref{eq:or}. 
We can now conclude: Comparing \eqref{eq:not_too_small_prob} with \eqref{eq:high_prob}, \eqref{eq:high_prob1} and \eqref{eq:high_prob2}, we see that
\begin{align}
\tEa\cap \tEb \cap \tEc \cap E^{\text{repair}}\cap E^{\text{conn}}\cap \{x^{**}=x\}
\label{eq:Eandx^**}
\end{align}
is non-empty for all $n$ large enough under the condition $(0,0)\leftrightarrow(2n,\Z^d)$ in $[\omegas,\omegae]_n$. 
If \eqref{eq:strategy} holds on~\eqref{eq:Eandx^**}, we can take logarithms and substitute the bounds from \eqref{eq:2n}--\eqref{eq:n2n} to obtain
\begin{align*}
a_{2n}^\beta\leq 2a_{n+\ell_n^4}^\beta+Cn^{1/2+\delta}\leq 2a_n^\beta+C'n^{\frac12+\delta}.
\end{align*}
It remains to verify \eqref{eq:strategy} on the event~\eqref{eq:Eandx^**}. Let $\pi$  be a path with $\pi(0)=0$ and $\pi(n)=x$. Using the definition of $E^{\text{repair}}$, we find $j_1\in[n-2\ell_n^2,n-\ell_n^2]$, $j_2\in [n+\ell_n^2,n+2\ell_n^2]$ and paths $\pi_1:[0,j_1]\to\Z^d$ and $\pi_2:=[j_2+2\ell_n^4,2n+2\ell_n^4]\to\Z^d$ such that 
\begin{align}
\pi(m)&=\pi_1(m) \text{ for }m\in[0,n-2\ell_n^2],\label{eq:preimage1}\\
\pi(m)&=\pi_2(m+2\ell_n^4) \text{ for }m\in[n+2\ell_n^2,2n],\label{eq:preimage2}\\
H_{(0,n]}([\omegas,\omegae]_n,\pi)&\geq H_{(0,j_1]}([\omegas,\omegasl,\omegae]_{n,n+2\ell_n^4},\pi_1),\label{eq:less1}\\
H_{(n,2n]}([\omegas,\omegae]_n,\pi)&\geq H_{(j_2+2\ell_n^4,2n+2\ell_n^4]}([\omegas,\omegasl,\omegae]_{n,n+2\ell_n^4},\pi_2)\label{eq:less2}
\end{align}
and such that 
\begin{align*}
[\omegas,\omegasl,\omegae]_{n,n+2\ell_n^4}\in C^{j_1,\pi_1(j_1)}_{n-j_1+\ell_n^4}\cap \accentset{\leftarrow}C^{j_2+2\ell_n^4,\pi_2(j_2+2\ell_n^4)}_{j_2-n+\ell_n^4}.
\end{align*}
\begin{figure}[h]
\includegraphics[width=\textwidth]{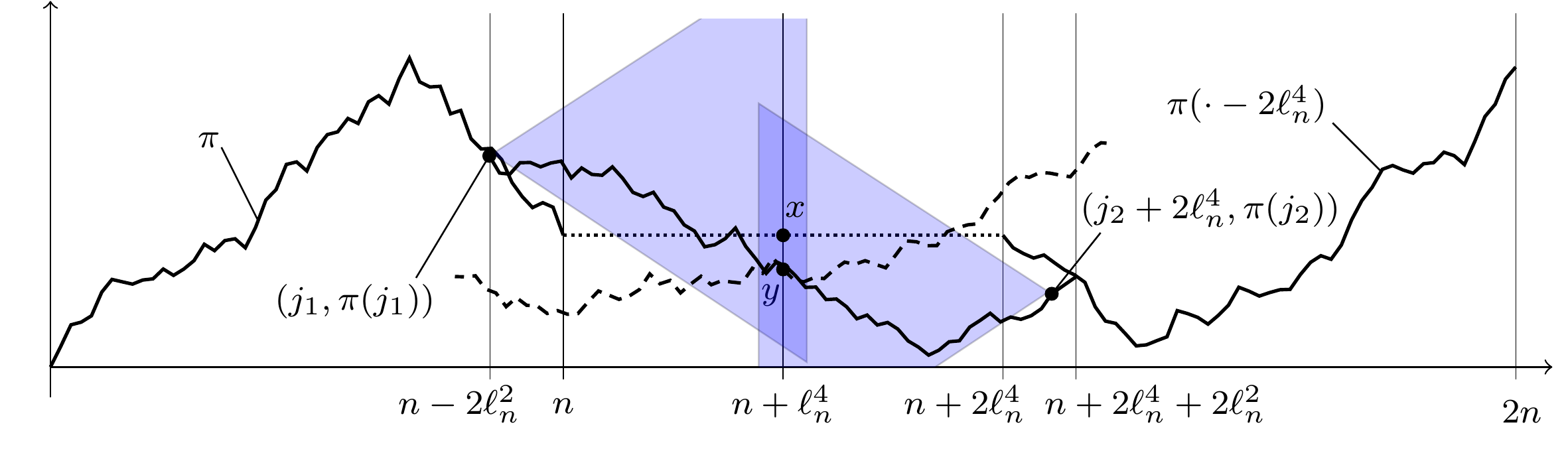} 
\caption{Illustration of the event on which \eqref{eq:strategy} is justified. The event $E^{\text{\,repair}}$ guarantees that we find points from which the coupled zone grows linearly (shaded) and which are not too far away from time $n$ and not too far away from $\pi$. The event $E^{\text{\,conn}}$ guarantees that we find a forward and backward percolation point $y$ at time $n$ in the intersection of cones. The repairing path is the solid path connecting $(j_1,\pi(j_1))$ and $(j_2+2\ell_n^4,\pi(j_2))$.}\label{fig:non_rand}
\end{figure}
See Figure \ref{fig:non_rand}. 
Let $y\in x+[-\ell_n^2,\ell_n^2]$ be as in \eqref{eq:def_connect}. The definition of the coupled zone \eqref{eq:def_coupled} implies that, in $[\omegas,\omegasl,\omegae]_{n,n+2\ell_n^4}$,
\begin{align*}
(j_1,\pi_1(j_1))\leftrightarrow (n+\ell_n^4,y)\leftrightarrow (j_2+2\ell_n^4,\pi_2(j_2+2\ell_n^4)).
\end{align*}
We can therefore extend $\pi_1$ and $\pi_2$ to a path $\pi':[0,2n+2\ell_n^4]$ using only open sites in $[j_1,j_2]$. Together with \eqref{eq:less1} and \eqref{eq:less2}, we have 
\begin{align*}
H_{2n}([\omegas,\omegae]_n,\pi)\geq H_{2n+2\ell_n^4}([\omegas,\omegasl,\omegae]_{n,n+2\ell_n^4},\pi')
\end{align*}
In addition, by \eqref{eq:preimage1} and \eqref{eq:preimage2}, at most $(2d+1)^{2\ell_n^2+2\ell_n^4}$ paths $\pi$ are mapped onto the same path $\pi'$, which finishes the proof.
\end{proof}

\begin{proof}[Proof of Theorem \ref{thm:NRF} for $\beta<\infty$]
Note that, since $\P((0,0)\leftrightarrow\infty)>0$ and since the almost sure limit in \eqref{eq:free} is determistic, we have
\begin{align*}
\lim_{n\to\infty}\frac{1}{n}\E[\log Z^\beta_n\mid(0,0)\leftrightarrow(n,\Z^d) ]=\lim_{n\to\infty}\frac{1}{n}\E[\log Z^\beta_n]=\f(\beta,p).
\end{align*}
Now the argument is the same as in zero temperature, with Lemmas \ref{lem:almost_superadd_pos} and \ref{lem:almost_subadd_pos} replacing Lemmas \ref{lem:almost_superadd} and \ref{lem:almost_subadd}. The constants in those lemmas are uniform, so the resulting bound is uniform as well.
\end{proof}

\appendix

\section{Comments on oriented percolation results}
In this section, we indicate where one can find proofs of the results listed in Section~\ref{sec:known}. All the proofs provide the uniform controls on the constants as in Theorems~\ref{thm:largefinite}--\ref{thm:large_initial}, though not always clearly stated. For readers' convenience, we make some comments on how to get the uniform bounds. Since what follows are commentaries to the references, we make no attempt to make it self-contained. In particular, we adopt the notations used therein. 

\subsection{On Theorem~\ref{thm:largefinite}}
\label{sec:5.1}
In dimension $d=1$ the proof is given in \cite[Section 12]{D84}. To get a uniform bound, we need $\gamma(\alpha(\pcr+\eps)/2)$ in \cite[(11.1)]{D84} to be positive uniformly in $p>\pcr+\eps$. The claim is first proved for $1$-dependent percolation with $p$ very close to $1$. It is then extended to the whole range of parameters using a renormalization construction that yields an embedded percolation process $\eta$ whose parameter is arbitrarily close to $1$. To make the bound uniform, one needs to verify that the spatial scale parameter $L$ can be chosen such that, uniformly in $p>\pcr+\eps$, the percolation parameter of the embedded process satisfies $\P(\eta(z)=1)>1-3^{-36/(1-q)}$, where $q=3/4$. This can be checked by noting that $\P(\eta(z)=1)$ is increasing in $p$.

In higher dimensions $d\geq 2$, the corresponding statement for the contact process is proved as the verification of the condition (b) in \cite[(5.2) Theorem]{D91}, and it can easily be adapted to the oriented percolation setting. The proof uses the same one-dimensional renormalization construction as above for which a uniform exponential tail is verified. This construction spends some time in ``looking for an occupied copy of $I=[-J,J]^d$''~\cite[fifth line on p.16]{D91}, where $J$ is a uniformly bounded constant arising from the renormalization. In order to get a uniform bound, we have to control this waiting time. But the waiting time is stochastically dominated by $J$ times a geometric random variable with parameter $\delta$, where in our notation, 
\begin{align*}
\delta=\P((0,0)\leftrightarrow(J,x)\text{ for all }x\in[-J,J]^d).
\end{align*}
Since this $\delta$ is increasing in $p$, it follows that the waiting time is longest for $p=\pcr+\epsilon$. 

\subsection{On Theorem \ref{thm:coupled}}
\label{sec:5.3}
The corresponding result for the contact process in random environment can be found in \cite[eq.~(43)]{GG12} and adaptation to our discrete time setting is routine. The constants $A,B,\gamma$ therein can be chosen uniformly in $\lambda\in\Lambda$, where $\Lambda=[\lambda_{\text{min}},\lambda_{\text{max}}]^{\Z^d}$ for $\lambda_{\text{cr}}(\Z^d)<\lambda_{\text{min}}<\lambda_{\text{max}}$. Indeed, after a little general arguments, it boils down to proving \cite[eq.~(46)]{GG12} but its right-hand side is already expressed only in terms of $\alpha$ and $\beta$ which are the contact processes with birth rate $\lambda_{\text{min}}$. The constant $\gamma$ in \cite[eq.~(43)]{GG12} corresponds to our $v$ in Definition~\ref{def:coupled}. 

\subsection{On Theorem \ref{thm:large_initial}}
The result for the contact process is proved in \cite[Theorem 2.30(b)]{L99}, based on the renormalization construction introduced in \cite{BG90, BG91}, and the adaptation to oriented percolation follows the same lines as in Section~\ref{sec:5.1}.

\section*{Acknowledgement}
The authors are grateful to the reviewer for very careful reading. RF was supported by ISHIZUE 2019 of Kyoto University Research Development Program and in part by KAKENHI 17H01093 and 18H03672. SJ was supported by a JSPS Postdoctoral Fellowship for Research in Japan, Grant-in-Aid for JSPS Fellows 19F19814.


\end{document}